\documentclass[11pt]{amsart}
\usepackage{amssymb,amsfonts,amsmath}
\usepackage{hyperref}
\usepackage{comment}

\usepackage[margin=1.2in]{geometry}

\usepackage[all,cmtip]{xy}
\usepackage{tikz}
\usetikzlibrary{matrix,arrows,decorations.pathmorphing,automata, matrix, positioning, calc, shapes.multipart}
\usepackage{tikz-cd}

\newcommand{\cC}{\mathcal{C}}
\newcommand{\cA}{\mathcal{A}}

\newcommand{\cM}{\mathcal{M}}
\newcommand{\cT}{\mathcal{T}}
\newcommand{\cD}{\mathcal{D}}
\newcommand{\cO}{\mathcal{O}}
\newcommand{\cI}{\mathcal{I}}
\newcommand{\cS}{\mathcal{S}}
\newcommand{\cH}{\mathcal{H}}
\newcommand{\cJ}{\mathcal{J}}
\newcommand{\cL}{\mathcal{L}}

\newcommand{\fh}{\mathfrak{h}}
\newcommand{\bv}{\mathbf{v}}
\newcommand{\bw}{\mathbf{w}}

\DeclareMathOperator{\DR}{DR}
\DeclareMathOperator{\Rep}{Rep}
\usepackage{mathrsfs}
\newcommand{\fX}{\mathfrak{X}}

\newcommand{\bA}{\mathbf{A}}

\newcommand{\fp}{\mathfrak{p}}

\newcommand{\eu}{\mathsf{eu}}

\newcommand{\CC}{\mathbb{C}}

\newcommand{\QQ}{\mathbb{Q}}
\newcommand{\RR}{\mathbb{R}}

\newcommand{\ZZ}{\mathbb{Z}}

\newcommand{\bc}{\mathbf{c}}

\newcommand{\fS}{\mathfrak{S}}

\newcommand{\fg}{\mathfrak{g}}

\newcommand{\C}{\mathbb{C}}

\DeclareMathOperator{\Ind}{Ind}
\DeclareMathOperator{\modd}{-mod}

\DeclareMathOperator{\GL}{GL}
\DeclareMathOperator{\gl}{\mathfrak{gl}}
\DeclareMathOperator{\Mat}{Mat}

\DeclareMathOperator{\Spec}{Spec}
\DeclareMathOperator{\Hom}{Hom}

\DeclareMathOperator{\gr}{gr}
\DeclareMathOperator{\End}{End}
\DeclareMathOperator{\triv}{triv}
\DeclareMathOperator{\Ext}{Ext}

\DeclareMathOperator{\ad}{ad}
\DeclareMathOperator{\Res}{Res}

\DeclareMathOperator{\sign}{sign}

\DeclareMathOperator{\rank}{rank}

\DeclareMathOperator{\HC}{HC}

\DeclareMathOperator{\Loc}{Loc}

\DeclareMathOperator{\opp}{opp}
\DeclareMathOperator{\Vect}{Vect}
\DeclareMathOperator{\Supp}{SS}

\DeclareMathOperator{\bimod}{-bimod}

\DeclareMathOperator{\id}{id}

\DeclareMathOperator{\rad}{rad}
\DeclareMathOperator{\Triv}{Triv}

\DeclareMathOperator{\im}{im}
\DeclareMathOperator{\Proj}{Proj}
\DeclareMathOperator{\sing}{SS}

\newcommand{\cP}{\mathcal{P}}
\newcommand{\scrM}{\mathscr{M}}
\newcommand{\cU}{\mathcal{U}}
\newcommand{\scrA}{\mathscr{A}}
\newcommand{\scrI}{\mathscr{I}}
\newcommand{\fP}{\mathfrak{P}}
\newcommand{\cB}{\mathcal{B}}
\newcommand{\gHC}{\mathcal{HC}}
\newcommand{\DD}{\mathbb{D}}
\newcommand{\euss}{\eu\operatorname{-ss}}

\newcommand{\homf}[2]{\Hom_{\mathtt{fin}}(#1,#2)}

\newtheorem{theorem}{Theorem}[section]
\newtheorem{remark}[theorem]{Remark}
\newtheorem{proposition}[theorem]{Proposition}
\newtheorem{definition}[theorem]{Definition}
\newtheorem{corollary}[theorem]{Corollary}
\newtheorem{lemma}[theorem]{Lemma}

\newtheorem{example}[theorem]{Example}

\newtheorem{conjecture}[theorem]{Conjecture}

\author{Jos\'e Simental}
\address{Instituto de Matem\'aticas, Universidad Nacional Aut\'onoma de M\'exico. Ciudad Universitaria, CDMX. M\'exico.}
\email{simental@im.unam.mx}

\title{Harish-Chandra bimodules for type A rational Cherednik algebras}

\begin{document}
	\begin{abstract}
		We study Harish-Chandra bimodules for the rational Cherednik algebra associated to the symmetric group $S_{n}$. In particular, we show that for any parameter $c \in \CC$, the category of Harish-Chandra $H_{c}$-bimodules admits a fully faithful embedding into the category $\cO_{c}$, and describe the irreducibles in the image. We also construct a duality on the category of Harish-Chandra bimodules, and in fact we do this in a greater generality of quantizations of Nakajima quiver varieties. We use this duality, along with induction and restriction functors, to describe, as an abelian category, the smallest Serre subcategory of the category of Harish-Chandra bimodules containing the regular bimodule, as well as to explicitly describe the tensor products of its irreducible objects.
	\end{abstract}

\maketitle

\section{Introduction}\label{sect:introduction}

Rational Cherednik algebras were introduced by Etingof and Ginzburg in \cite{EG}. Since then, they have been studied in connection with several areas of mathematics, such as integrable systems, algebraic geometry, knot homology and representation theory. Among the rational Cherednik algebras, perhaps the most studied class is the one associated to the symmetric group $S_{n}$ acting on $\CC^{n}$ by permuting the coordinates. The objective of this paper is to study Harish-Chandra bimodules over these and more general algebras. 

Let us give the definition of a HC bimodule, in general. Let $A$ be a $\ZZ_{\geq 0}$-graded algebra, finite over its center $Z$ and let $\mathcal{A}$ be a quantization of $A$, that is, $\mathcal{A}$ comes equipped with a $\ZZ_{\geq 0}$-filtration $F$ and an isomorphism $\gr\mathcal{A} \cong A$. We also assume that $Z$ is a finitely generated algebra, in particular, both $A$ and $\cA$ are Noetherian. An $\cA$-bimodule $\cB$ is called \emph{Harish-Chandra} (HC, for short) if it admits a bimodule filtration so that the left and right actions of the commutative algebra $Z$ on $\gr\cB$ coincide, making $\gr\cB$ a \emph{finitely generated} module over $Z$. Note that it automatically follows that $\cB$ is a finitely generated $\cA$-bimodule. 

Examples of such algebras $\mathcal{A}$ abound in representation theory, and HC bimodules have been studied over them in many cases. For example, $\mathcal{A}$ could be the universal enveloping algebra of a semisimple Lie algebra.  The theory of HC bimodules over such algebras arises in the study of translation functors and has been well-studied in \cite{bernsteingelfand, gbimod}. Another example arises when $\mathcal{A}$ is (a central reduction of) a finite W-algebra. In this context, HC bimodules have been studied for example in \cite{losev_w, ginzburg_w}. As a more general example (that also encompasses the type $A$ rational Cherednik algebras), $A$ is the algebra of functions of a conical symplectic singularity, and $\cA$ is its quantization, \cite{BPW, BLPW, BL}. 

Or, one can take $\cA$ to be a rational Cherednik algebra, we will recall the definition of these in Section \ref{subsect:defrca}. In this context, HC bimodules have been introduced by Berest, Etingof and Ginzburg in \cite{BEG_do} and their study has been further developed in \cite{losev_completions, losev_derived, hcbimod}. In particular in \cite{hcbimod}, the author gave a complete characterization of the category of HC bimodules modulo those without full support, which turns out to be equivalent to the category of representations of a group that can be explicitly recovered from the data defining the rational Cherednik algebra. In that paper, the author also counted the number of irreducible HC bimodules in the type $A$ case, that is, for $W = S_{n}$ acting on $\CC^n$ by permuting the coordinates. The first main result of this paper is a much subtler refinement of \cite[Theorem 6.8]{hcbimod}. Recall that the type $A$ rational Cherednik algebra depends on a parameter $c \in \CC$, we will denote by $\HC(c,c)$ the category of HC bimodules associated to the algebra with parameter $c$.  Also, recall that there is a (extensively studied, see for example \cite{GGOR, losev_derived, shan, rouquier}) highest weight category $\cO_{c}$ of modules over the rational Cherednik algebra, whose definition will be reviewed in Section \ref{subsect:O}.

\begin{theorem}\label{thm:main}
For any $c \in \CC \setminus \RR_{< 0}$, the functor $\Phi_{c}: \HC(c,c) \to \cO_{c}$ defined by $\Phi_{c}(B) = B \otimes_{H_{c}} \Delta(\triv)$ satisfies the following properties.
\begin{enumerate}
\item $\Phi_{c}$ is exact, fully faithful and intertwines the restriction functors.
\item The image of $\Phi_{c}$ is closed under subquotients.
\item $\Phi_{c}$ preserves the filtrations by support. More precisely, if the support of $B \in \HC(c,c)$ is $\{v \in (\CC^{n} \oplus \CC^{n})/S_{n} : W' \subseteq W_{v}\}$ for some parabolic subgroup $W'$ determined by $B$, then the support of $\Phi_{c}(B)$ is $\{x \in \CC^n : W' \subseteq W_{x}\}$. 
\item The simples in the image of $\Phi_{c}$ are given as follows:
\begin{itemize}
\item[(a)] If $c \not\in  \{r/m : \gcd(r;m) = 1, 1 \leq m \leq n\}$ then the only simple in the image of $\Phi_{c}$ is $L_{c}(\triv) = \Delta_{c}(\triv)$.
\item[(b)] If $c \in \ZZ$, then $\Phi_{c}$ is a category equivalence.
\item[(c)] If $c = r/m > 0$ with $\gcd(r;m) = 1$ and $1 < m \leq n$, then a simple $L_{c}(\lambda)$ belongs to the image of $\Phi_{c}$ if and only if it has the form $\lambda = \mu + m\nu$, where $\mu = ((m-1)^{x}, y)$ for some $x \geq 0$ and $0 \leq y < m-1$. Here, addition of partitions is done component-wise, and so is multiplication of partitions by scalars. 
\end{itemize}
\end{enumerate}
\end{theorem}

Some remarks are in order. First, if $c \in \RR_{<0}$ the functor $\bullet \otimes_{H_{c}}\Delta(\sign)$ satisfies analogous properties. Second, note that Theorem \ref{thm:main} is a Cherednik algebra analogue of the corresponding result for semisimple Lie algebras, \cite[Theorem 5.9]{bernsteingelfand}. In the Cherednik setting, however, an analogue of Theorem \ref{thm:main} does not exist outside of type $A$: there are simply \emph{too many} irreducible HC bimodules for it to hold. For example, if $W = \ZZ/\ell\ZZ$ then category $\cO_{c}$ always has $\ell$ irreducibles, while it is not difficult to find values of the parameter $c$ for which $\HC(c,c)$ has $(\ell - 1)^2 + 1$ irreducibles. 

A final remark is that, while Theorem \ref{thm:main} gives a precise count of the number of irreducibles in $\HC(c,c)$ it does \emph{not} describe the category $\HC(c,c)$ completely, for the image of $\Phi_{c}$ is not in general closed under extensions. We do not know how to describe extensions between arbitrary simples. However, we can completely describe the Serre subcategory spanned by the simple subquotients of the regular bimodule $H_{c}$. Note that this category is not semismiple only in case (c) of Theorem \ref{thm:main}(4). 

\begin{theorem}\label{thm:main2}
Denote by $\mathcal{P}_{r/m}$ the minimal Serre subcategory of $\HC(r/m, r/m)$ containing the regular bimodule $H_{r/m}$. Then, $\mathcal{P}_{r/m}$ is equivalent to the category of representations of the quiver

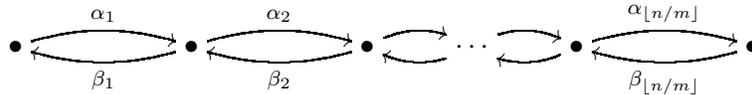
\begin{equation*}\label{eqn:quiver}
\xymatrix{1\ar@/^/[rr]^{\alpha_{1}} & & 2 \ar@/^/[ll]^{\beta_{1}} \ar@/^/[rr]^{\alpha_{2}} & & 3 \ar@/^/[ll]^{\beta_{2}} \ar@/^/[r] & \cdots \ar@/^/[l] \ar@/^/[r] & \ar@/^/[l] \lfloor n/m\rfloor  \ar@/^/[rr]^{\alpha_{\lfloor n/m\rfloor}} & & \lfloor n/m\rfloor + 1  \ar@/^/[ll]^{\beta_{\lfloor n/m\rfloor}}}  \tag{$\dagger$}
\end{equation*}

\noindent with relations $\alpha_{i}\beta_{i} = \beta_{i}\alpha_{i} = 0$ for every $i = 1, \dots, \lfloor n/m\rfloor$. 
\end{theorem}

The category $\mathcal{P}_{r/m}$ is quadratic, but not Koszul, and it has infinite homological dimension. In the case $n = 2$, the category $\mathcal{P}_{r/m}$ actually coincides with the category of all HC bimodules. In particular, the category of HC bimodules is not, in general, a highest weight category. 





In the proof of Theorem \ref{thm:main2}, we will need some more structure on the category $\HC(c,c)$: (co-)induction and restriction functors that were introduced in \cite{losev_completions} and studied further in \cite{losev_derived, hcbimod} as well as the present paper; and a duality functor whose existence is our last main result. The first indication that such a duality functor could exist was given in \cite{BL}, where the authors of that paper introduce a bimodule (the \emph{double wall-crossing bimodule}) whose composition series looks like that of the regular bimodule but the order in which the subquotients appear is reversed. It turns out that the double wall-crossing bimodule is the dual of the regular bimodule.

\begin{theorem}\label{thm:main3}
Let $c > 0$. Then, for an integer $N \gg 0$ there exists an equivalence of categories

$$
\DD: \HC(c,c) \to \HC(-c + N, -c + N)^{\opp}
$$

\noindent that preseves the support of a bimodule. If $c = r/m$, then $\DD$ induces an equivalence $\cP_{r/m} \to \cP_{-r/m + N}^{\opp}$.
\end{theorem} 

In the context of Theorem \ref{thm:main2}, the duality $\DD$ preseves every vertex, and it behaves on arrows as $\alpha_{i} \leftrightarrow \beta_{i}$, $i = 1, \dots, \lfloor n/m\rfloor$. We remark, however, that the duality $\DD$ does \emph{not} intertwine the tensor products.

The proof of Theorem \ref{thm:main3} is geometric, and it uses the realization of (the spherical subalgebra of) the rational Cherednik algebra as a quantized quiver variety. For these algebras, we have both a notion of geometric and algebraic HC bimodules, the former are sheaves on a resolution of the quiver variety. That such a duality exists on a geometric level is not surprising since HC bimodules, by their definition, satisfy a condition analogous to holonomicity of D-modules. The content of Theorem \ref{thm:main3} is, then, to show that such an equivalence still exists at the algebraic level. Note that, even if abelian localization holds at $c$, it may not hold for the category $\HC(c,c)$. 

The organization of this paper is as follows. In Section \ref{sect:prelim} we recall the definition of rational Cherednik algebras and their categories $\cO$, as well as the main properties of the Bezrukavnikov-Etingof restriction and induction functors, \cite{BE}, that will be important in this paper. In Section \ref{sect:hcbimod} we recall known results about categories of Harish-Chandra bimodules for general rational Cherednik algebras, including induction and restriction functors from \cite{losev_completions}. In this section we also give a proof that the category of HC bimodules is equivalent to the category of representations of a finite-dimensional algebra that, although well-known to experts, does not seem to have been explicitly written in the literature. We specialize to type $A$ in Section \ref{sect:HCvsO}, where we prove Theorem \ref{thm:main} and use it to give some explicit computations of induction and restriction functors for HC bimodules, as well as to give an alternative description of the double wall-crossing bimodule. Section \ref{sect:duality} deals with the proof of Theorem \ref{thm:main3}. Since we prove this theorem in the more general context of quantized quiver varieties, a large portion of this section is on preliminaries of these. To prove this theorem, we have to introduce a generalized category of HC bimodules, that are roughly speaking those that are supported on the graph of an automorphism of the quiver variety. Finally, in Section \ref{sect:prblock} we use the computations of restriction and induction functors taken up in Section \ref{sect:HCvsO}, as well as properties of the duality functor, to prove Theorem \ref{thm:main2}.\\

{\bf Acknowledgements.} I would like to thank Ivan Losev, Dmytro Matvieievskyi and Vasileios Aravantinos-Sotiropoulos for helpful discussions. I am also grateful to Dmytro Matvieievskyi for very helpful comments on a preliminary version of this paper. Finally, I am very grateful to an anonymous referee for a careful reading of this manuscript, indicating several inaccuracies and making suggestions that allowed me to improve the exposition, and catching a significant gap in a previous version of the text. 

\section{Rational Cherednik algebras and category $\cO$}\label{sect:prelim}

\subsection{Rational Cherednik algebras}\label{subsect:defrca}

Let $W$ be a complex reflection group, with reflection representation $\fh$. By $S \subseteq W$ we will denote the set of reflections of $W$, that is, $s \in S$ if and only if $\rank(\id_{\fh} - s) = 1$. For $s \in S$, let $\alpha^{\vee}_{s} \in \fh$ be an eigenvector with eigenvalue $\lambda_{s} \neq 1$, and $\alpha_{s} \in \fh^{*}$ an eigenvector with eigenvalue $\lambda_{s}^{-1}$. The elements $\alpha_{s}, \alpha_{s}^{\vee}$ are determined up to a nonzero scalar, and we partially eliminate this ambiguity by requiring $\langle \alpha_{s}, \alpha_{s}^{\vee}\rangle = 2$. 

Note that $W$ acts on $S$ by conjugation, and we let $c: S \to \CC$ be a $W$-invariant function, the set of such functions forms a vector space $\fp \cong \CC^{|S/W|}$. For $c \in \fp$ the rational Cherednik algebra $H_{c} = H_{c}(W, \fh)$ is, by definition, the quotient of the semidirect product $T(\fh \oplus \fh^{*}) \rtimes W$ by the relations
\[
[x, x'] = [y, y'] = 0, [y,x] = \langle y, x\rangle - \sum_{s \in S}c(s)\langle y, \alpha_{s}\rangle\langle \alpha_{s}^{\vee}, x\rangle s, \hspace{1cm} x, x' \in \fh^{*}, y, y' \in \fh
\]

Note that $H_{c}$ is the specialization of an algebra $H_{\fp}$ defined over $\fp$ -- the relations are the same, but the parameters $c(s)$ are formal variables $\bc(s)$ instead of complex numbers. For most of this paper, we will focus on the case $W = S_{n}$, acting on $\fh := \{(x_1, \dots, x_n) \mid \sum x_{i} = 0\} \subseteq \CC^n$ by permuting the coordinates.

\begin{example}\label{def:rca}
	Let us take $W = S_{n}$ acting on $\fh$ by permuting the coordinates. Here, $|S/W| = 1$ and so the parameter may be thought of as a single complex number $c \in \CC$.  The rational Cherednik algebra $H_{c} = H_{c}(n)$ is the quotient of the semidirect product algebra $\CC\langle x_{1}, \dots, x_{n}, y_{1}, \dots, y_{n}\rangle\rtimes S_{n}$ by the relations	
	\[
	\begin{array}{lllll}
	\sum_{i = 1}^{n} x_{i} = 0, & \sum_{i = 1}^{n} y_{i} = 0, &[x_{i}, x_{j}] = 0, & [y_{i}, y_{j}] = 0, & [y_{i}, x_{j}] = cs_{ij} - \frac{1}{n}
	\end{array}
	\]

\noindent where, in the last equation, we have $i \neq j$. 
\end{example}

The algebra $H_{c}$ can be filtered by setting $\deg(\fh^*) = 1 = \deg(\fh)$ while $\deg(W) = 0$. The PBW theorem for Cherednik algebras says that, under this filtration, we have a natural isomorphism $\gr H_{c} \cong S(\fh \oplus \fh^*) \rtimes W$. This implies that $H_{c}$ admits a triangular decomposition $H_{c} \cong S(\fh^*) \otimes \CC W \otimes S(\fh)$, where $S(\fh^*)$, $\CC W$ and $S(\fh)$ all sit naturally as subalgebras of $H_{c}$. 

Finally, note that the algebra $H_{c}$ also admits a grading by setting $\deg(\fh^{*}) = 1, \deg(\fh) = -1$ and $\deg(W) = 0$. This grading is inner: define the Euler element $\eu \in H_{c}$ by $\eu := \frac{1}{2}\sum y_{i}x_{i} + x_{i}y_{i}$, where $y_{1}, \dots, y_{n}$ is a basis of $\fh$ and $x_{1}, \dots, x_{n}$ the dual basis (it is straightforward to check that $\eu$ does not depend on the chosen basis). Then, $[\eu, x] = x, [\eu, y] = -y$ and $[\eu, w] = 0$ for $x \in \fh^{*}$, $y \in \fh$ and $w \in W$. 

\subsection{Category $\cO_{c}$}\label{subsect:O} Following \cite{GGOR} we define the category $\cO_{c}$ to consist of those $H_{c}$-modules that are finitely generated and on which $S(\fh)_{+}$ acts locally nilpotently. Note that any module $M \in \cO_{c}$ is actually finitely generated over the algebra $S(\fh^*)$. 

Analogously to the representation theory of semisimple Lie algebras, we can define standard modules as follows. Let $\lambda$ be an irreducible representation of $W$.  We may see $\lambda$ as a representation of $S(\fh)\rtimes W$ by letting $\fh$ act by $0$ on $\lambda$. We now define the standard module to be
\[
\Delta(\lambda) := \Ind_{S(\fh)\rtimes W}^{H_{c}}\lambda = H_{c}\otimes_{S(\fh) \rtimes W}\lambda
\]

It is easy to see that the standard module $\Delta(\lambda)$ has a unique irreducible quotient, which we denote $L(\lambda)$. Moreover, the modules $L(\lambda)$ where $\lambda$ runs over all irreducible representations of $W$ form a complete and non-redundant list of irreducible objects of the category $\cO_{c}$.

The category $\cO_{c}$ has the structure of a highest weight category. This means that it has enough projectives and there is a partial order on the set of irreducibles that satisfies some upper triangularity properties. Let us denote by $P(\lambda)$ the projective cover of $L(\lambda)$, so we have homomorphisms $P(\lambda) \twoheadrightarrow \Delta(\lambda) \twoheadrightarrow L(\lambda)$. A consequence of $\cO_{c}$ being a highest weight category is that the sets $\{[P_{c}(\lambda)]\}, \{[\Delta_{c}(\lambda)]\}, \{[L_{c}(\lambda)]\}$ where $\lambda$ runs over the irreducible representations of $W$ are all bases of $K_{0}(\cO_{c})$. 
 
Let $W' \subseteq W$ be a parabolic subgroup, that is, $W'$ is the stabilizer of a point $b \in \fh$. Bezrukavnikov and Etingof \cite{BE} have introduced induction and restriction functors relating the categories $\cO_{c}(W, \fh)$ and $\cO_{c}(W', \fh_{W'})$, where $\fh_{W'}$ is the unique $W'$-stable complement to $\fh^{W}$ in $\fh$. These functors depend on a point $b \in \fh$  with $W' = W_{b}$, but for different such $b$ the functors are (not canonically!) isomorphic, so we will simply denote them $\Res^{W}_{W'}: \cO_{c}(W) \to \cO_{c}(W')$, $\Ind^{W}_{W'}: \cO_{c}(W') \to \cO_{c}(W)$. The following encompasses important properties of these functors that we will need.

\begin{theorem}\label{thm:prop res and ind}
The following is true.

\begin{enumerate}
\item The functors $\Res^{W}_{W'}, \Ind^{W}_{W'}$ are exact and biadjoint, \cite[Proposition 2.9]{shan}, \cite{losev_isofunctors}.
\item On the Grothendieck group level, $[\Res^{W}_{W'}\Delta(\lambda)] = \sum \dim(\Hom_{W'}(\tau, \lambda))[\Delta(\tau)]$, where the sum on the right-hand side runs over irreducible representations of $W'$, \cite[Proposition 3.14]{BE}.  
\item Moreover, $\Res^{W}_{W'}\Delta(\lambda)$ admits a filtration whose subquotients are all standard modules, cf. \cite[Proposition 1.9]{shan}.
\end{enumerate}
\end{theorem}

The following easy corollary will be important.

\begin{corollary}\label{cor:respoly}
	Let $W' \subseteq W$ be a parabolic subgroup. Then, $\Res^{W}_{W'}\Delta(\triv) = \Delta(\triv_{W'})$. 
\end{corollary}
\begin{proof}
Follows immediately from Theorem \ref{thm:prop res and ind}, (2) and (3). 
\end{proof}

We remark that Theorem \ref{thm:prop res and ind} (2) and (3), and therefore Corollary \ref{cor:respoly}, remain valid when replacing the standard module $\Delta(\triv)$ with the costandard module $\nabla(\triv)$. This follows from the compatibility of the restriction functors with the naive duality functors that interchange standard and costandard modules, which is the main result of \cite{losev_isofunctors}. 

We will also consider the category $\cO_{c}^{\euss}$, which is the full subcategory of $\cO_{c}$ consisting of modules where the Euler element $\eu \in H_{c}$ acts diagonalizably. We remark that it is always the case that $\Delta(\lambda), L(\lambda) \in \cO_{c}^{\euss}$, but projectives need not be $\eu$-semisimple. In particular, $\cO_{c}^{\euss} \subseteq \cO_{c}$ has the same simple objects, but it may have less extensions. We note that $\cO_{c}^{\euss}$ has enough projectives. This stems from the following result.

\begin{lemma}
Let $P \in \cO_{c}$ be any module, and assume that $N, M \subseteq P$ are submodules such that $P/N, P/M \in \cO_{c}^{\euss}$. Then, $P/(M\cap N) \in \cO_{c}^{\euss}$. In particular, $P$ has the largest quotient $P^{\euss}$ which belongs to $\cO^{\euss}$.
\end{lemma}
\begin{proof}
Note that the obvious map $P/(M\cap N) \to P/M \oplus P/N$ is injective. Since $P/M \oplus P/N \in \cO^{\euss}$, the result follows.
\end{proof}

\begin{corollary}
Let $P(\lambda) \in \cO_{c}$ be the projective cover of $S(\lambda)$. Then, $P(\lambda)^{\euss}$ is projective in $\cO_{c}^{\euss}$ and surjects onto $S(\lambda)$. Thus, the category $\cO^{\euss}$ has enough projectives.
\end{corollary}
\begin{proof}
Since $L(\lambda) \in \cO_{c}^{\euss}$ and $P(\lambda)^{\euss}$ is the largest quotient of $P(\lambda)$ which belongs to $\cO_{c}^{\euss}$, it is clear that we have a surjection $P(\lambda)^{\euss} \twoheadrightarrow L(\lambda)$. Now assume that we have a surjection $M \twoheadrightarrow N$ with $M, N \in \cO_{c}^{\euss}$ and a map $P(\lambda)^{\euss} \to N$. This gives a map $P(\lambda) \to N$ and, since $P(\lambda)$ is projective, this map lifts to $P(\lambda) \to M$. But $M \in \cO^{\euss}$, so this latter map factors through $P(\lambda)^{\euss}$. It follows that $P(\lambda)^{\euss}$ is projective in $\cO_{c}^{\euss}$. 
\end{proof}

We note, however, that $\cO_{c}^{\euss}$ need not be a highest weight category. For example, when $n = 2$ and $c = 1/2$, the category $\cO_{c}^{\euss}$ is equivalent to the category of representations of the quiver
\begin{equation*}\label{eqn:quiver}
\xymatrix{1\ar@/^/[rr]^{\alpha} & & 2 \ar@/^/[ll]^{\beta}}
\end{equation*}
with relations $\alpha\beta = \beta\alpha = 0$. This category has infinite homological dimension, and thus cannot be highest weight.

\section{Harish-Chandra bimodules}\label{sect:hcbimod}

\subsection{Definition and basic properties.}\label{subsect:defhc} Note that, independently of the parameter $c \in \fp$, the algebras $S(\fh)^{W}$ and $S(\fh^{*})^{W}$ may be embedded inside the algebra $H_{c}$. So, for two parameters $c, c' \in \fp$ and a $(H_{c}, H_{c'})$-bimodule $B$, it makes sense to consider the operator $\ad(a): B \to B$,  $b \mapsto ab - ba$ for any $a \in S(\fh^{*})^{W} \cup S(\fh)^{W}$. 

\begin{definition}
	Let $c, c' \in \fp$. A $(H_{c}, H_{c'})$-bimodule $B$ is said to be \emph{Harish-Chandra} (shortly, \emph{HC}) if $B$ is finitely generated and for every $a \in S(\fh^{*})^{W} \cup S(\fh)^{W}$, the operator $\ad(a)$ on $B$ is locally nilpotent. 
\end{definition}

 For example, the algebra $H_{c}$ itself with diagonal action is a HC $H_{c}$-bimodule. We will refer to this as the \emph{regular bimodule}. Note that its sub-bimodules are precisely the two-sided ideals of $H_{c}$. Also, finite-dimensional $(H_{c}, H_{c'})$-bimodules, if they exist, are HC. We will denote by $\HC(c,c')$ the category of Harish-Chandra $(H_{c}, H_{c'})$-bimodules. We remark that subquotients and extensions of HC bimodules are HC. Below, we will see that $\HC(c,c')$ is actually equivalent to the category of modules over a finite-dimensional algebra. The following are basic results on HC bimodules.

\begin{lemma}\label{lemma:basicbimod}
Let $B \in \HC(c,c')$. Then:
\begin{enumerate}
\item $B$ is finitely generated as a left $H_{c}$-module, as a right $H_{c'}$-module, and as a $S(\fh)^{W} \otimes S(\fh^{*})^{W \opp}$-module, where the superscript \lq$\opp$\rq \, means that the corresponding algebra acts on the right, \cite[Lemma 3.3]{BEG_do}. 
\item If $B' \in \HC(c', c'')$, then $B \otimes_{H_{c'}} B' \in \HC(c,c'')$, \cite[Lemma 8.3]{BEG_do}.
\item If $M \in \cO_{c'}$, then $B \otimes_{H_{c'}} M \in \cO_{c}$, \cite[Proposition 3.1]{losev_derived}.
\item The bimodule $B$ has finite length, \cite[Corollary 6.3.3]{ginzburg}
\end{enumerate}
\end{lemma}

Let us provide a way to construct HC bimodules. Let $M \in \cO_{c}$, $N \in \cO_{c'}$. Then, $\Hom_{\CC}(N, M)$ becomes a $(H_{c}, H_{c'})$-bimodule. Let $\homf{N}{M}$ denote the sum of the HC sub-bimodules of $\Hom_{\CC}(N, M)$. A priori, $\homf{N}{M}$ is only a pro-object of $\HC(c,c)$, but \cite[Proposition 5.7.1]{losev_completions} shows that $\homf{N}{M}$ is finitely generated, so it is HC. Moreover, every irreducible HC bimodule $B$ can be embedded into a bimodule of the form $\homf{N}{M}$ for suitable $N \in \cO_{c'}$, $M \in \cO_{c}$, cf. \cite[Proposition 2.4]{hcbimod}. Furthermore, both $N$ and $M$ can be taken to be irreducible.

Every module $N \in \cO_{c'}$ gives a functor $\homf{N}{\bullet}: \cO_{c} \to \HC(c,c')$ that is right adjoint to the tensor product functor $\bullet \otimes_{H_{c'}} N: \HC(c,c') \to \cO_{c}$. The following result should be well-known to experts, but we could not find a proof for it in the literature.

\begin{theorem}\label{thm: module category}
For any parameters $c, c'$, the category $\HC(c,c')$ is equivalent to the category of modules over a finite-dimensional algebra.
\end{theorem}
\begin{proof}
We have to check:
\begin{itemize}
\item[(a)] That every HC bimodule has finite length.
\item[(b)] That $\HC(c,c')$ has a finite number of irreducible objects.
\item[(c)] That every irreducible can be embedded into an injective object.
\end{itemize}

Statement (a) is (4) of Lemma \ref{lemma:basicbimod}. For statement (b), recall that every irreducible bimodule can be embedded into a bimodule of the form $\homf{N}{M}$ for suitable \emph{irreducible} modules $N \in \cO_{c'}, M \in \cO_{c}$. By (a), $\homf{N}{M}$ has finite length. So the number of irreducible HC  $(H_{c}, H_{c'})$-bimodules is bounded above by $\sum_{\lambda, \mu}\text{length}(\homf{L_{c'}(\lambda)}{L_{c}(\mu)})$ and (b) follows. We move on to (c). Let $P_{c'}$ be a projective generator of $\cO_{c'}$. Thanks to \cite[Lemma 3.9]{losev_derived}, the functor $\bullet \otimes_{H_{c'}} P_{c'}$ is exact, so its right adjoint $\homf{P_{c'}}{\bullet}$ preserves injective objects. Now let $B \in \HC(c,c')$ be irreducible. As in \cite[Lemma 3.10]{losev_derived}, we can check that $B \otimes_{H_{c'}} P_{c'} \neq 0$, and so we have an embedding $B \hookrightarrow \homf{P_{c'}}{B \otimes_{H_{c'}} P_{c'}}$. If $E \in \cO_{c}$ is injective admitting an embedding $B \otimes_{H_{c'}} P_{c'} \hookrightarrow E$, then $\homf{P_{c'}}{B \otimes_{H_{c'}}P_{c'}} \hookrightarrow \homf{P_{c'}}{E}$. The latter bimodule is injective in $\HC(c,c')$, and this finishes the proof of the theorem. 
\end{proof}

\subsection{Supports of Harish-Chandra bimodules} We have the following result of Losev that tells us that the notion of a HC bimodule given in Section \ref{subsect:defhc} coincides with the one given in the introduction. Recall that both algebras $H_{c}$ and $H_{c'}$ admit a filtration whose associated graded is the smash-product algebra $S(\fh \oplus \fh^{*}) \rtimes W$. This algebra is finite over its center $Z$, which can be identified with the algebra of invariants $S(\fh \oplus \fh^{*})^{W} \cong \C[\fh \oplus \fh^{*}]^{W}$.

\begin{theorem}[Section 5.4, \cite{losev_completions}]\label{thm:altdef}
An $(H_{c}, H_{c'})$-bimodule $B$ is HC if and only if it admits a bimodule filtration such that $\gr B$ is a finitely generated $S(\fh \oplus \fh^{*})\rtimes W$-bimodule, and the left and right actions of $Z$ on $\gr B$ coincide. 
\end{theorem}

We can use Theorem \ref{thm:altdef} to define the singular support of a HC bimodule as follows. Let $B \in \HC(c,c)$. Pick a filtration on $B$ as in Theorem \ref{thm:altdef}. Then, we let $\sing(B) \subseteq \Spec Z = (\fh \oplus \fh^{*})/W$ be the set theoretic support of the module $\gr B$. By a standard argument, $\sing(B)$ does not depend on the chosen filtration. 

We remark that, in general, $\sing(B)$ is a union of symplectic leaves of the Poisson variety $(\fh \oplus \fh^{*})/W$. The symplectic leaves in $(\fh \oplus \fh^{*})/W$ are in correspondence with conjugacy classes of parabolic subgroups as follows. Let $W' \subseteq W$ be a parabolic subgroup. Then, $\mathcal{L}_{W'} := \pi\{(x,y) \in \fh \oplus \fh^{*} : W_{(x,y)} = W'\}$ is a symplectic leaf, where $\pi: \fh \oplus \fh^{*} \to (\fh \oplus \fh^{*})/W$ is the projection, and all symplectic leaves arise in this way, cf. \cite[Proposition 7.4]{browngordon}. Note that it follows, first, that there are only a finite number of symplectic leaves and, second, that $\overline{\mathcal{L}_{W'}} = \bigcup_{W' \subseteq W''} \mathcal{L}_{W''}$. For a symplectic leaf $\mathcal{L}$, we will denote by $\HC_{\overline{\mathcal{L}}}(c, c')$ the category of HC $(H_{c}, H_{c'})$-bimodules whose singular support is contained in $\overline{\mathcal{L}}$. We also denote $\HC_{\partial \mathcal{L}}(c,c')$ the full subcategory consisting of bimodules whose singular support is contained in $\overline{\mathcal{L}} \setminus \mathcal{L}$. Note that this is a Serre subcategory in $\HC_{\overline{\mathcal{L}}}(c,c')$ and we denote the quotient $ \HC_{\overline{\mathcal{L}}}(c,c')/\HC_{\partial \mathcal{L}}(c,c')$ by $\HC_{\mathcal{L}}(c,c')$.

\subsection{Restriction and induction functors for Harish-Chandra bimodules}

In this section we recall the properties of restriction and induction functors for categories of HC bimodules, that were first introduced by Losev in \cite{losev_completions} and further studied in \cite{losev_derived},\cite{hcbimod}.

Let $\underline{W} \subseteq W$ be a parabolic subgroup, and denote by $\underline{H}_{c}$ the rational Cherednik algebra associated to $\underline{W}$, with reflection representation $\fh_{\underline{W}}$, that is the unique $\underline{W}$-stable complement to $\fh^{\underline{W}}$ in $\fh$. Note that the normalizer $N := N_{W}(\underline{W})$ setwise stabilizes $\fh_{\underline{W}}$, and from here it is easy to see that it acts on $\underline{H}_{c}$. Let us denote $\Xi := N_{W}(\underline{W})/\underline{W}$. A $\Xi$-equivariant HC $(\underline{H}_{c}, \underline{H}_{c'})$-bimodule $B$ is a HC bimodule together with an action of $N$ on $B$, such that the action of $\underline{W} \subseteq N$ coincides with the adjoint action, $\underline{w}.b = \underline{w}b\underline{w}^{-1}$. For example, the regular bimodule $\underline{H}_{c}$ is $\Xi$-equivariant. We will denote by $\underline{\HC}^{\Xi}(c,c')$ the category of $\Xi$-equivariant HC $(\underline{H}_{c}, \underline{H}_{c'})$-bimodules. 

In \cite[Section 3]{losev_completions}, Losev constructs a restriction functor $\bullet_{\dagger}: \HC(c,c') \to \underline{\HC}^{\Xi}(c,c')$ with the following properties, cf. \cite[Theorem 3.4.6]{losev_completions}.

\begin{itemize}
\item[(R1)] $\bullet_{\dagger}$ is exact and intertwines the tensor product.
\item[(R2)] $B_{\dagger} = 0$ if and only if $\sing(B)\cap \cL_{\underline{W}} = \emptyset$.
\item[(R3)] If $B$ is irreducible, then $B_{\dagger}$ is finite-dimensional if and only if $\sing(B) = \overline{\cL_{\underline{W}}}$.
\item[(R4)] $\bullet_{\dagger}$ admits a right adjoint $\bullet^{\dagger}: \underline{\HC}^{\Xi}(c, c') \to \HC(c,c')$ that is left exact.
\item[(R5)] If $B \in \HC_{\overline{\cL_{\underline{W}}}}(c,c')$, then the kernel and cokernel of the adjunction map $B \to (B_{\dagger})^{\dagger}$ belong to $\HC_{\partial\cL_{\underline{W}}}(c,c')$. 
\end{itemize}

We remark that in \cite{losev_completions} it is only shown that $B^{\dagger}$ is the direct limit of its HC sub-bimodules. Below, in Lemma \ref{lemma:indfg} we will see that $B^{\dagger}$ is indeed finitely generared and therefore HC. 

The construction of the functors $\bullet_{\dagger}, \bullet^{\dagger}$ is quite technical. In computations, we will need a more simple minded version of these functors that are constructed in a similar spirit to the restriction functors for category $\cO_{c}$ but do not see the $\Xi$-equivariance. These functors are denoted by $\bullet_{\heartsuit}: \HC(c,c) \to \underline{\HC}(c,c')$, $\bullet^{\heartsuit}: \underline{\HC}(c,c') \to \HC(c,c')$ . They satisfy properties analogous to (R1)-(R4) above. The relationship between $\bullet_{\dagger}, \bullet^{\dagger}$ and $\bullet_{\heartsuit}, \bullet^{\heartsuit}$ is as follows.

\begin{lemma}[Lemma 3.10.1 in \cite{losev_completions}]\label{lemma:equivariant vs naive}
The following holds.

\begin{enumerate}
\item If $B \in \HC(c,c')$, then $B_{\heartsuit} = \operatorname{Forget}(B_{\dagger})$, where $\operatorname{Forget}$ is the functor that forgets the $\Xi$-equivariance.
\item  If $B \in \underline{\HC}^{\Xi}(c,c')$, then there is a functorial embedding $B^{\dagger} \hookrightarrow (\operatorname{Forget}(B))^{\heartsuit}$.
\end{enumerate}
\end{lemma}

We remark that the embedding in Lemma \ref{lemma:equivariant vs naive}(2) is in general a \emph{proper} embedding. We also remark that the property (R5) above is \emph{not} valid for the functors $\bullet_{\heartsuit}, \bullet^{\heartsuit}$ and this is the main reason why we need the functors $\bullet_{\dagger}$ and $\bullet^{\dagger}$. Similarly to the functor $\bullet^{\dagger}$, in \cite{losev_completions} it is only shown that $B^{\heartsuit}$ is the direct limit of its HC sub-bimodules. In the next result, we show that it is indeed HC. 

\begin{lemma}\label{lemma:indfg}
Let $B \in \underline{\HC}(c,c')$. Then, $B^{\heartsuit}$ is a finitely generated $(H_{c}, H_{c'})$-bimodule and therefore it is HC. In particular, $B^{\dagger} \subseteq B^{\heartsuit}$ is also HC. 
\end{lemma}
\begin{proof}
Since we know that $B^{\heartsuit}$ is the direct limit of its HC subbimodules and, by Theorem \ref{thm: module category}, $\HC(c,c')$ is equivalent to the category of finite-dimensional representations of some finite-dimensional algebra, we need to show that every simple HC bimodule can only appear a finite number of times as a composition factor of $B^{\heartsuit}$. This is equivalent to showing that, for every projective HC $(H_{c}, H_{c'})$-bimodule $P$, $\Hom_{\HC}(P, B^{\heartsuit})$ is finite-dimensional. By adjunction, $\Hom_{\HC}(P, B^{\heartsuit}) = \Hom_{\underline{\HC}}(P_{\heartsuit}, B)$. The latter hom space is finite-dimensional because both bimodules $P_{\heartsuit}$ and $B$ are HC. We are done. 
\end{proof}

To finish this subsection, we recall a result of \cite[Section 3.3]{losev_derived} that tells us that the restriction functors for HC bimodules and category $\cO_{c}$ are compatible in an obvious way.

\begin{lemma}\label{lemma:comp res}
Let $B \in \HC(c, c')$ and $M \in \cO_{c'}$. Then, $\Res^{W}_{\underline{W}}(B \otimes_{H_{c'}} M) \cong B_{\heartsuit} \otimes_{\underline{H}_{c}} \Res^{W}_{\underline{W}}(M)$. 
\end{lemma}

\section{Harish-Chandra bimodules and category $\cO_{c}$}\label{sect:HCvsO}

In this section, we prove Theorem \ref{thm:main}. From now on, we will only deal with the rational Cherednik algebra of the symmetric group $S_{n}$, the parameter here is a single complex number $c \in \CC$. The strategy to prove this theorem is quite simple. We first deal with the case $H_{r/n}(n)$. For this case, the category of Harish-Chandra bimodules has been described in \cite{hcbimod} and so we can actually prove the theorem \lq\lq by hand.\rq\rq \, Then, we use restriction functors to prove the theorem in more general cases. 

\subsection{Case $c = r/n$}\label{subsect:r/n} First, we recall the description of the category $\HC(c,c)$ that was given in \cite{hcbimod}. In this case, the algebra $H_{c}$ has a unique proper two-sided ideal, $J$, and the quotient $Q := H_{c}/J$ is finite-dimensional. By \cite[Proposition 2.11]{hcbimod}, $H_{c}$ coincides with the injective hull of $J$ in the category $\HC(c,c)$. There is one more indecomposable bimodule, namely the \emph{double wall-crossing bimodule} $D$. The results in \cite[Section 6.5]{hcbimod} say that $H_{c}, D, Q$ and $J$ exhaust the list of isomorphism classes of indecomposable HC bimodules; that $H_{c}$ is both the injective hull of $J$ and the projective cover of $Q$; and that $D$ is the injective hull of $Q$ and the projective cover of $J$. 

Recall now that we denote by $\Phi_{c}$ the functor $\HC(c,c) \rightarrow \cO_{c}$, $B \mapsto B \otimes_{H_{c}}\Delta_{c}(\triv)$, with right adjoint $\Psi_{c}: \cO_{c} \rightarrow \HC(c,c)$, $M \mapsto \homf{\Delta_{c}(\triv)}{M}$. We remark that $\Phi_{c}$ is exact by \cite[Lemma 3.9]{losev_derived}. Note that $\Phi_{c}(H_{c}) = \Delta_{c}(\triv)$ and $\Phi_{c}(Q) = L_{c}(\triv)$, the unique finite-dimensional irreducible module in category $\cO_{c}$. By exactness of $\Phi_{c}$, it follows that $\Phi_{c}(J) = I:= \rad(\Delta_{c}(\triv))$. In particular, $\Phi_{c}$ does not kill any simple bimodule. The faithfulness of $\Phi_{c}$ now follows from the fact that an exact functor that does not kill nonzero objects must be faithful. Note that, by standard properties of adjunctions, it follows that the unit of adjunction $B \rightarrow \Psi_{c}\Phi_{c}(B)$ is injective for any HC bimodule $B$. Also, we remark that $\homf{\Delta_{c}(\triv)}{L_{c}(\triv)} = \homf{L_{c}(\triv)}{L_{c}(\triv)} = \Hom_{\CC}(L_{c}(\triv), L_{c}(\triv)) = Q$. The first equality follows because $\homf{I}{L_{c}} = 0$, as $I$ and $L_{c}$ are simple modules with different support, cf. \cite[Lemma 5.7.2]{losev_completions}. The next-to-last equality follows because $L_{c}(\triv)$ is finite-dimensional.

Let us now show that $\Phi_{c}$ is full. For this, we compute $\Phi_{c}(D)$. By exactness, $\Phi_{c}(D)$ must be an extension of $L_{c}(\triv)$ by $I$. This extension cannot be trivial, since $D$ embeds into $\Psi_{c}\Phi_{c}(D)$. Thus, we conclude that $\Phi_{c}(D) = \nabla_{c}(\triv)$. But now it follows from \cite[Corollary 4.7 and Proposition 6.12]{hcbimod} that the unit of adjunction $\id_{\HC(c,c)} \rightarrow \Psi_{c}\Phi_{c}$ is actually an isomorphism. Thus, $\Phi_{c}$ is fully faithful. The claim about the image being closed under subquotients follows immediately by the computations on this section, as well as the claim about the compatibility of supports.


\subsection{General case} Let us now show Theorem \ref{thm:main} for general $c$. First, we note that we can restrict to a certain set of parameters.

\begin{lemma}
Fix $n > 0$, and let $\HC(c,c)$ denote the category of HC bimodules for the rational Cherednik algebra of the symmetric group $S_n$ with its reflection representation.
\begin{enumerate}
\item If $c \not\in \{r/m \mid 1 \leq m \leq n \; \text{and} \; \gcd(r;m) = 1\}$ then $\HC(c,c)$ is equivalent, as a tensor category, to the category of vector spaces. The unique simple is the regular bimodule $H_{c}$.
\item If $c \in \ZZ$, then $\HC(c,c)$ is equivalent, as a tensor category, to the category of representations of the symmetric group $S_n$.
\end{enumerate}
Moreover, in these cases Theorem \ref{thm:main} holds. 
\end{lemma}
\begin{proof}
In both cases (a) and (b), the category $\cO_{c}$ is semisimple and all modules have full support, cf. e.g. \cite{Wilcox}. Then both parts follow from \cite[Theorem 1.1]{hcbimod}. In Case (a), Theorem \ref{thm:main} holds trivially, while in Case (b) it is a consequence of \cite[Theorem 8.15]{BEG_do}.
\end{proof}

Thus, for the rest of this section we will assume that $c = r/m$ with $1 < m \leq n$ and $\gcd(r;m) = 1$. Moreover, we may restrict to $c > 0$. If $c < 0$, we simply replace $\Delta(\triv)$ by $\Delta(\sign)$. Let us remark that exactness of $\Phi_{c}$ follows from \cite[Lemma 3.9]{losev_derived}.

\begin{remark}\label{rmk: faithful res}
Under the assumption that $c = r/m > 0$, it follows from \cite[Lemma 4.2]{losev_Bernstein} and \cite[Theorem 5.8.1]{losev_completions} that the possible supports of HC bimodules are as follows. For each $i = 0, 1, \dots, \ell := \lfloor n/m\rfloor$, let $\mathcal{L}_{i} \subseteq (\fh \oplus \fh^{*})/S_n$ be the symplectic leaf $\cL_{S_{m}^{\times i}}$ (so that, for example, $\mathcal{L}_{0}$ is the open symplectic leaf.) Then, the possible supports of a HC bimodule are $\overline{\cL_{0}}, \dots, \overline{\cL_{\ell}}$. Note that
\[
\overline{\cL}_{\ell} \subseteq \overline{\cL}_{\ell - 1} \subseteq \cdots \subseteq \overline{\cL}_{0} = (\fh \oplus \fh^{*})/S_n.
\]
Note, in particular, that the restriction functor $\bullet_{\dagger^{S_n}_{S_m^{\ell}}}$ does not kill any nonzero HC bimodules. We will repeatedly use this below. 
\end{remark}

\subsubsection{$\Phi_{c}$ is faithful} Let us remark that, since $\Phi_{c}$ is exact, to check that it is faithful it is enough to show that $\Phi_{c}$ does not kill any nonzero bimodules.

\begin{lemma}
	Let $c = r/m> 0$, $1 < m \leq n$, $\gcd(r;m) = 1$. Let $B \in \HC(c,c)$ be nonzero. Then, $B \otimes_{H_{c}} \Delta_{c}(\triv)$ is nonzero.
\end{lemma}
\begin{proof}
Let $\ell := \lfloor n/m\rfloor$ and consider $W' = S_{m}^{\times \ell}$, a parabolic subgroup of $W = S_{n}$. By our choice of $\ell$, the bimodule $B_{\heartsuit^{W}_{W'}}$ is nonzero, cf. Remark \ref{rmk: faithful res}. Thus, by the results of Section \ref{subsect:r/n}, $B_{\heartsuit^{W}_{W'}} \otimes_{H_{c}(W')} \Delta_{c}(\triv_{W'})$ is nonzero. But this is $\Res^{W}_{W'}(B \otimes_{H_{c}} \Delta_{c}(\triv))$, cf. Corollary \ref{cor:respoly} and Lemma \ref{lemma:comp res}. The result follows.	
\end{proof}

\subsubsection{$\Phi_{c}$ is full} Now we show that $\Phi_{c}$ is full. In order to do this, we will show that the unit map $\id_{\HC(c, c)} \rightarrow \Psi_{c}\Phi_{c}$ is an isomorphism where, recall, $\Psi_{c}: \cO_{c} \rightarrow \HC(c, c)$ is $\homf{\Delta_{c}(\triv)}{\bullet}$. Since $\Phi_{c}$ is faithful, the unit is always injective.

 \begin{lemma}\label{lemma:full}
	Let $c= r/m > 0$, $1 < m \leq n$ and $\gcd(r;m) = 1$. Let $B \in \HC(c, c)$. Then, the adjunction map $B \rightarrow \homf{\Delta_{c}(\triv)}{B \otimes_{H_{c}}\Delta_{c}(\triv)}$ is an isomorphism.
\end{lemma}
\begin{proof}
	Denote $N := \Phi_{c}(B) \in \cO_{c}$. Since $\Phi_{c}$ is exact and faithful, to check that the map $B \rightarrow \homf{\Delta_{c}(\triv)}{N}$ is an isomorphism, it is enough to check that $\homf{\Delta_{c}(\triv)}{N} \otimes_{H_{c}} \Delta_{c}(\triv)$ is isomorphic to $N$. Note that, since $N$ is in the image of $\Phi_{c}$, we have an epimorphism $\homf{\Delta_{c}(\triv)}{N} \otimes_{H_{c}}\Delta_{c}(\triv) \twoheadrightarrow N$. We show that its kernel is zero. To do so, let $\ell := \lfloor n/m\rfloor$ and let $W' := S_{m}^{\times \ell} \subseteq S_{n}$.  Consider $\Res^{S_{n}}_{W'}$, which is an exact functor and, by our choice of $\ell$, does not kill nonzero modules. So it is enough to check that the induced epimorphism
	
	\begin{equation}\label{eqn:ind_epi}
	\Res^{S_{n}}_{W'}(\homf{\Delta_{c}(\triv)}{N} \otimes_{H_{c}}\Delta_{c}(\triv)) \twoheadrightarrow \Res^{S_{n}}_{W'}(N)
	\end{equation}
	
	\noindent is an isomorphism, i.e. it is injective. Recall that we have $\Res^{S_{n}}_{W'}(\homf{\Delta_{c}(\triv)}{N} \otimes_{H_{c}} \Delta_{c}(\triv)) = \homf{\Delta_{c}(\triv)}{N}_{\heartsuit^{S_{n}}_{W'}} \otimes_{H_{c}(W')} \Delta_{c}(\triv_{W'})$. By construction of the restriction functor, we have that $\homf{\Delta_{c}(\triv)}{N}_{\heartsuit^{S_{n}}_{W'}}$ embeds into $\homf{\Delta_{c}(\triv_{W'})}{\Res^{S_{n}}_{W'}(N)}$, so by exactness of the functor $\bullet \otimes_{H_{c}(W')} \Delta_{c}(\triv_{W'})$, it is enough to show that the morphism
	
	$$
	\homf{\Delta_{c}(\triv_{W'})}{\Res^{S_{n}}_{W'}(N)} \otimes_{H_{c}(W')} \Delta_{c}(\triv_{W'}) \rightarrow \Res^{S_{n}}_{W'}(N)
	$$
	
	\noindent is an isomorphism. Since $H_{c}(W') = H_{c}(m)^{\otimes \ell}$, this follows by the results of Section \ref{subsect:r/n}, provided we show that $\Res^{S_{n}}_{W'}(N)$ belongs to the image of the functor $\bullet \otimes_{H_{c}(W')} \Delta_{c}(\triv_{W'})$. This follows because the map in (\ref{eqn:ind_epi}) is an epimorphism and from results of Section \ref{subsect:r/n}, namely, that for $c = r/m$ the image of $\Phi_{c}$ is closed under subquotients. We are done.
\end{proof}

\subsubsection{The image of $\Phi_{c}$ is closed under subquotients.} Let us now show that the image of the functor $\Phi_{c}$ is closed under subquotients. First, we show that it is closed under quotients. 

\begin{lemma}\label{lemma:closedquotients}
	Let $c= r/m > 0$, $1 < m \leq n$ and $\gcd(r;m) = 1$. Let $M := B \otimes_{H_{c}}\Delta_{c}(\triv)$ for some $B \in \HC(H_{c}, H_{c})$, and assume that $f: M \twoheadrightarrow N$ is an epimorphism. Then, there exists a quotient $B'$ of $B$ such that $B' \otimes_{H_{c}}\Delta_{c}(\triv) = N$.
\end{lemma}
\begin{proof}
	Note that by Lemma \ref{lemma:full} we may assume that $B = \homf{\Delta_{c}(\triv)}{M}$. Consider the induced morphism
	\[
	B = \homf{\Delta_{c}(\triv)}{M} \buildrel \overline{f} \over \rightarrow \homf{\Delta_{c}(\triv)}{N}
	\]
	
	\noindent and let us denote by $B'$ its image. So we have a map $B' \otimes_{H_{c}}\Delta_{c}(\triv) \rightarrow N$ which is an epimorphism since $f$ is an epimorphism. Let us show that it is injective. So let $W' := S_{m}^{\times \ell} \subseteq S_{n}$ where $\ell := \lfloor n/m \rfloor$. It is enough to check that the induced epimorphism $\Res^{S_{n}}_{W'}(B' \otimes_{H_{c}} \Delta_{c}(\triv)) \rightarrow \Res^{S_{n}}_{W'}(N)$ is an isomorphism. This follows exactly as in the proof of Lemma \ref{lemma:full}.
\end{proof}

Let us now show that the image of $\Phi_{c}$ is closed under submodules. So let $M = B \otimes_{H_{c}} \Delta_{c}(\triv)$ and $N' \subseteq M$. Let $N = M/N'$. By Lemma \ref{lemma:closedquotients} we know that there exists a quotient $B'$ of $B$ such that $B' \otimes_{H_{c}}\Delta_{c}(\triv) = N$. If we let $B''$ be the kernel of the epimorphism $B \rightarrow B'$, exactness of $\Phi_{c}$ shows that $N' = B'' \otimes_{H_{c}} \Delta_{c}(\triv)$.

Finally, we remark that (3) of Theorem \ref{thm:main} is a consequence of Lemma 2.6 in \cite{hcbimod}. To finish the proof of Theorem \ref{thm:main} (minus assertion (4) there), it remains to show that $\Phi_{c}$ intertwines the restriction functors. This follows at once from Corollary \ref{cor:respoly} and Lemma \ref{lemma:comp res}. Let us state an easy consequence of Theorem \ref{thm:main}.

\begin{corollary}\label{cor:losevetingofstoica}
	Assume $c \not\in \RR_{< 0}$. For any $B \in \HC(H_{c}, H_{c})$, there is an order-preserving bijection between sub-bimodules of $B$ and submodules of $\Phi_{c}(B)$. In particular, there is an order-preserving bijection between the set of two-sided ideals of $H_{c}$ and the set of submodules of $\Delta_{c}(\triv)$, which is given by $\cJ \mapsto \cJ\Delta_{c}(\triv)$.
\end{corollary}

The statement about ideals in $H_{c}$ and submodules of the polynomial representation $\Delta_{c}(\triv)$ is not new. The set of ideals of $H_{c}$ was calculated by Losev in \cite[Theorem 5.8.1]{losev_completions}, while the set of submodules of the standard module $\Delta_{c}(\triv)$ was calculated by Etingof-Stoica in \cite[Theorem 5.10]{etingof-stoica}. However, the proof we give here is closer in spirit to that of the corresponding result for universal enveloping algebras, see for example \cite[Section 3.4]{gbimod}.

\subsubsection{$\Phi_c$ preserves filtration by supports} The statement (3) of Theorem \ref{thm:main} follows from \cite[Lemma 2.6]{hcbimod}, that we reproduce here for convenience.

\begin{lemma}\label{lem: support compatibility}
Let $B \in \HC(c,c)$ be an irreducible HC bimodule, and let $N \in \cO_{c}$ be irreducible. Then, $B \otimes_{H_{c}} N = 0$ unless $\mathrm{RAnn}(B) = \mathrm{Ann}(N)$. If $B \otimes_{H_{c}} N \neq 0$, then the annihilator of every irreducible quotient of $B \otimes_{H_{c}} N$ coincides with $\mathrm{LAnn}(B)$. 
\end{lemma}

Now we show Theorem \ref{thm:main} (3). Let $c= r/m > 0$, $1 < m \leq n$ and $\gcd(r;m) = 1$. If $B \in \HC(c,c)$ is irreducible, then there must exist a simple subquotient $S$ of $\Delta_{c}(\triv)$ such that $B \otimes_{H_{c}} S \neq 0$. If $\Supp(B) = \overline{\cL}_{i}$, then by Lemma \ref{lem: support compatibility}, this simple subquotient has to be the unique simple subquotient of $\Delta_{c}(\triv)$ (cf. \cite[Theorem 5.10]{etingof-stoica}) whose support is $\overline{\{x \in \C^{n} \mid S_{m}^{\times i} \subseteq W_{x}\}}$. This implies (3) of Theorem \ref{thm:main}. 

\subsection{Image} 

Let us describe the simples in the image of the functor $\Phi_{c}$. This will finish the proof of Theorem \ref{thm:main}. First of all, by a \emph{partition} $\lambda = (\lambda^{1}, \lambda^{2}, \dots)$ we mean an infinite nonincreasing sequence of nonnegative integers, with the property that only a finite number of the $\lambda_{i}$ are nonzero. We write $|\lambda| := \sum_{i} \lambda^{i}$, and call this the \emph{size} of $\lambda$. For an integer $m > 0$, a partition $\lambda$ is said to be \emph{m-restricted} if $\lambda^{i} - \lambda^{i - 1} < m$ for every $i$. Note that by the division algorithm, for every partition $\lambda$ there exist unique partitions $\mu, \nu$ with the property that $\mu$ is $m$-restricted and  $\lambda = \mu + m\nu$, where addition of partitions is done component-wise and so is multiplication by scalars. For example, if $\lambda = (7,5,1,1,0,\dots)$ and $m = 3$, then $\mu = (4,2,1,1,0,\dots)$ while $\nu = (1, 1, 0, \dots) $. 

We say that a partition $\nu$ is \emph{m-trivial} if it has the form $\nu = (m-1, m-1, \dots, m-1, b, 0, \dots)$ for some $0 \leq b < m - 1$ (note that it is possible that $\nu$ does not have parts equal to $m-1$, in which case $\nu$ has the form $(b, 0, \dots)$ for $b < m-1$.) Obviously, for each non-negative integer $k$ there exists a unique $m$-trivial partition $\mu$ of size $k$, which we denote by $\Triv_{m}(k)$.  

Let us explain the reason behind the terminology $m$-trivial. Let $q$ be a primitive $m$-root of unity, and consider the Hecke algebra $\cH_{q}(k)$. The algebra $\cH_{q}(k)$ is cellular and, as such, it has distinguished representations called \emph{Specht} modules, which are indexed by partitions $\lambda$ of size $k$, let us denote by $S(\lambda)$ the corresponding Specht module. The Specht modules are not, in general, irreducible, but they come equipped with a canonical bilinear form $\beta_{\lambda} : S(\lambda) \otimes S(\lambda) \rightarrow \CC$. The module $D(\lambda) := S(\lambda)/\rad(\beta_{\lambda})$ is now either irreducible or zero, and those nonzero modules $D(\lambda)$ form a complete and irredundant list of irreducible $\cH_{q}(k)$-modules. It is known that $D(\lambda)$ is nonzero if and only if $\lambda$ is $m$-restricted. Moreover, $D(\Triv_{m}(k))$ is precisely the trivial representation of $\cH_{q}(k)$. 

Let us go back to the setting of rational Cherednik algebras. The following is an easy consequence of Theorem 6.8 and Lemma 6.1 in \cite{hcbimod}.

\begin{proposition}
	Let $\lambda$ be a partition of $n$, and consider its decomposition $\lambda = \mu + m\nu$, where $\mu$ is $m$-restricted. Then, the irreducible $L_{c}(\lambda)$ belongs to the image of the functor $\Phi_{c}$ if and only if $\mu$ is $m$-trivial. 
\end{proposition} 

\begin{corollary}
	The irreducible HC $H_{c}$-bimodules are naturally labeled by partitions of $n$, $\lambda = \mu + m\nu$, where $\mu$ is $m$-trivial. For such a partition $\lambda$, we denote by $B(\lambda) = \homf{\Delta_{c}(\triv)}{L_{c}(\lambda)}$ the corresponding irreducible bimodule.
\end{corollary}

So we have an explicit description of the simples in the image of $\Phi_{c}$. We remark, however, that the image of $\Phi_{c}$ is, in general, \emph{not} closed under extensions. In fact, we have the following result.

\begin{lemma}
Let $B \in \HC(c,c)$. Then, $\Phi_{c}(B) \in \cO_{c}^{\eu\text{-}ss}$, that is, the Euler element acts diagonalizably on $\Phi_{c}(B)$. Moreover, the eigenvalues of $\eu$ on $\Phi_{c}(B)$ belong to the set \[\frac{n}{2} - c\frac{(n-1)n}{2} + \ZZ.\]
\end{lemma}
\begin{proof}
Note that $\eu$ acts diagonalizably on $\Delta(\triv)$. By \cite[Proposition 3.8]{BEG_do}, $\eu$ acts ad-diagonalizably, with integer eigenvalues, on $B$. It then follows easily that $\eu$ acts diagonalizably on $B \otimes_{H_{c}} \Delta(\triv)$, with eigenvalues belonging to the set \[\{b + \ZZ : b \; \text{is an eigenvalue of} \; \eu \; \text{on}  \; \Delta(\triv)\}.\] It is known that the eigenvalues of $\eu$ on $\Delta(\triv)$ belong to the set $\frac{n}{2} - c\frac{(n-1)n}{2} + \ZZ_{\geq 0}$. The result follows.
\end{proof}

\begin{example}
Let $n = 2$ and $c = 1/2$. It follows from results of Section \ref{subsect:r/n} that in this case $\HC(c,c) \cong \cO_{c}^{\eu\text{-}ss}$. We remark that $P(\sign) \not\in \cO_{c}^{\eu\text{-ss}}$, and so $\cO_{c}^{\eu\text{-ss}} \subsetneq \cO_{c}$. In particular, in this case the image of $\Phi_{c}$ is not closed under extensions. \\
More generally, for $n \geq 2$, if $c = r/n$ with $\gcd(r;n) = 1$, then it follows from Section \ref{subsect:r/n} that $\HC(c,c)$ is equivalent to the Serre span of $L_{c}(\triv)$ and $L_{c}(n-1, 1)$ \underline{in the category} $\cO_{c}^{\eu\text{-ss}}$, see e.g. \cite[Theorem 1.3]{BEG_fd}.
\end{example}

Let us denote by $\cC_{c}$ the image of the functor $\Phi_{c}$. Of course, $\cC_{c} \cong \HC(c,c)$, but we prefer to think of $\cC_{c}$ as living inside category $\cO_{c}$. 

\begin{conjecture}\label{conj:Serre}
The category $\cC_{c}$ is the Serre span \underline{in the category} $\cO_{c}^{\eu\text{-ss}}$ of the simples it contains. 
\end{conjecture}

\subsection{Restriction and induction functors}

Now we proceed to study the behaviour of restriction and induction under the functor $\Phi_{c}$ and prove that restriction has both a left and a right adjoint that, however, are not isomorphic. 

\subsubsection{Restriction functors} We assume that $c > 0$. For a parabolic subgroup $\underline{W} \subseteq S_{n}$, we will denote by $\cC_{c}(\underline{W}) \subseteq \cO_{c}(\underline{W})$ the image of the functor $\Phi_{c}^{\underline{W}}: \underline{\HC}(c,c) \to \cO_{c}(\underline{W})$, $B \mapsto B \otimes_{\underline{H}_{c}} \Delta_{c}(\triv_{\underline{W}})$. 

\begin{proposition}\label{prop:intertwine res}
	Let $c > 0$ and $\underline{W} \subseteq S_{n}$ a parabolic subgroup. Then, the functors $\Res^{S_{n}}_{\underline{W}} \circ \Phi_{c}^{S_{n}}$ and $\Phi_{c}^{\underline{W}}\circ \bullet_{\heartsuit^{S_{n}}_{\underline{W}}}: \HC(H_{c}(n)) \rightarrow \cO_{c}(\underline{W})$ are naturally isomorphic.  
\end{proposition}
\begin{proof}
Follows immediately from Corollary \ref{cor:respoly} and Lemma \ref{lemma:comp res}.
\end{proof}

Note that an immediate consequence of Proposition \ref{prop:intertwine res} is that if $M \in \cC_{c}$, then $\Res^{W}_{\underline{W}}(M) \in \cC_{c}(\underline{W})$. Since the restriction functors on category $\cO$ and on HC bimodules are identified via the functors $\Phi$, we will sometimes abuse the notation and denote $M_{\heartsuit^{S_{n}}_{\underline{W}}}$ for $\Res^{S_{n}}_{\underline{W}}(M)$ when $M \in \cC_{c}$. 

\subsubsection{Induction functors}

The compatibility of $\Phi_{c}$ with induction functors is more subtle. Let us denote by $\iota: \cC_{c} \rightarrow \cO_{c}$ the inclusion. Note that, since $\cC_{c}$ is closed under subquotients, $\iota$ admits a right adjoint, that we will denote by $\iota^{!}: \cO_{c} \rightarrow \cC_{c}$. In simple terms, $\iota^{!}(M)$ is the largest submodule of $M$ that belongs to the category $\cC_{c}$. 

\begin{proposition}\label{prop:compind}
	Let $\underline{W} \subseteq S_{n}$ be a parabolic subgroup, and $c \in \CC \setminus \RR_{<0}$. Then, the functors $\Phi_{c}^{S_{n}} \circ \bullet^{\heartsuit_{\underline{W}}^{S_{n}}}$ and $\iota^{!} \circ \Ind_{\underline{W}}^{S_{n}} \circ \Phi_{c}^{\underline{W}}:  \underline{\HC}(c,c) \rightarrow \cC_{c}(S_{n})$ are naturally isomorphic. 
\end{proposition}
\begin{proof}
By its construction, the functor $\Phi_{c}^{S_{n}} \circ  \bullet^{\heartsuit^{S_{n}}_{\underline{W}}}$ is right adjoint to the functor $\bullet_{\heartsuit^{S_{n}}_{\underline{W}}} \circ \Phi_{c}^{-1}: \cC(S_{n}) \to \underline{\HC}(c,c)$. Since $\Ind^{S_{n}}_{\underline{W}}$ is right adjoint to $\Res^{S_{n}}_{\underline{W}}$, Proposition \ref{prop:intertwine res} implies that the same is true for the functor $\iota^{!} \circ \Ind_{\underline{W}}^{S_{n}} \circ \Phi_{c}^{\underline{W}}$. The result follows.
\end{proof}

Similarly to the case of restriction, for $M \in \cC_{c}(\underline{W})$, we will write $M^{\heartsuit^{S_{n}}_{\underline{W}}}$ for the largest submodule of $\Ind^{S_{n}}_{\underline{W}}(M)$ that belongs to the category $\cC_{c}$. 



\begin{example}
Let us see that, in general, the functor $\iota^{!}$ in Proposition \ref{prop:compind} is really necessary and, in particular, $\Phi_{c}$ does \underline{not} commute with induction functors. Let $c \in \CC \setminus \QQ$. In particular, for every $n \geq 0$, the category $\HC(H_{c}(n))$ is equivalent to the category of vector spaces, with a unique simple object given by the regular bimodule, and $\cC_{c}(n) := \cC_{c}(S_{n})$ consists of objects which are isomorphic to direct sums of copies of the polynomial representation $\Delta_{c}(\triv)$. Now, it is well-known that in this case $\Ind^{S_{n}}_{\{1\}}(\CC) = \bigoplus_{\lambda \vdash n}\Delta_{c}(\lambda) \otimes \lambda$, which shows that $B^{\dagger^{S_{n}}_{\{1\}}} \otimes_{H_{c}(n)} \Delta_{c}(\triv)$ is a proper submodule of $\Ind^{S_{n}}_{\{1\}}(B \otimes_{H_{c}(0)} \Delta_{c}(\triv))$, where $B$ is the unique irreducible HC $H_{c}(0)$-bimodule. 
\end{example}


Let us now construct a left adjoint to the functor $\bullet_{\heartsuit}: \cC_{c}(S_{n}) \to \cC_{c}(\underline{W})$. The inclusion functor $\iota: \cC_{c} \to \cO_{c}$ also admits a left adjoint $^{!}\!\iota: \cO_{c} \to \cC_{c}$, namely $^{!}\!\iota(M)$ is the largest quotient of $M$ belonging to category $\cC_{c}$. So the functor $^{!}\!\iota \circ \Ind^{S_{n}}_{\underline{W}}: \cC_{c}(\underline{W}) \to \cC_{c}(S_{n})$ is left adjoint to $\bullet_{\heartsuit}$. We denote this functor by $^{\heartsuit}\!\bullet$. The construction of the functors $\bullet^{\heartsuit}$ and $^{\heartsuit}\!\bullet$ is schematized in the following diagram.

$$
\xymatrix{& & & & \cC_{c}(S_{n}) \\ \cC_{c}(\underline{W}) \ar[rr]^{\Ind} \ar@/^/[urrrr]^{\bullet^{\heartsuit}} \ar@/_/[drrrr]_{^{\heartsuit}\!\bullet} & & \cO_{c}(S_{n}) \ar[drr]_{^{!}\!\iota} \ar[urr]^{\iota^{!}} & & \\ & & & & \cC_{c}(S_{n})}
$$

A few words about terminology are in order. In \cite{losev_completions}, the functor $\bullet^{\heartsuit}$ is called simply the \emph{induction functor}, while the functor $^{\heartsuit}\!\bullet$ is not considered. The left adjoint to restriction is usually called induction, while the right adjoint is called coinduction. It is for this reason that, differing from \cite{losev_completions}, we will call the functor $\bullet^{\heartsuit}$ the \emph{coinduction} functor, while \emph{induction} will be reserved for $^{\heartsuit}\!\bullet$. As we will see below, in general induction and coinduction are not isomorphic. 

To finish this section, let us remark that if $c$ is integral then, in fact, $\cC_{c} = \cO_{c}$, and so in this case $\Phi_{c}$ does commute with induction, and induction and coinduction coincide. In particular, neither the induction nor the coinduction functor intertwine tensor products.

\subsubsection{Computations of restriction, induction and coinduction} Here, we will do some basic computations regarding the induction, coinduction and restriction functors. First, let us define an important bimodule.

\begin{definition}
Assume $c \not\in \RR_{< 0}$. We call the bimodule $D_{c} := \homf{\Delta_{c}(\triv)}{\nabla_{c}(\triv)} \in \HC(c,c)$ the \emph{double wall-crossing bimodule.}
\end{definition}

Note that if $c \not\in \{r/m \mid 1 < m \leq n, \gcd(r;m) = 1\}$ then in fact $\nabla_{c}(\triv) \cong \Delta_{c}(\triv)$, so in this case $D_{c}$ is isomorphic to the regular bimodule. So, most of the time, we will assume that $c = r/m > 0$, $\gcd(r;m) = 1$, and $1 < m \leq n$. 

The reason for the terminology is that, as we will see below, this bimodule coincides with the bimodule constructed by Bezrukavnikov-Losev in \cite[Section 7.6]{BL} (we remark that here we are referring to v3 of that paper on the arXiv, which differs from the published version) that is obtained as the tensor product of two wall-crossing bimodules. 

\begin{lemma}\label{lemma:BL}
Assume $c = r/m > 0$, with $\gcd(r;m) = 1$ and $1 < m \leq n$. The bimodule $D_{c}$ coincides with the double wall-crossing bimodule constructed in \cite{BL}.
\end{lemma}
\begin{proof}
First, we claim that in this case the costandard module $\nabla_{c}(\triv)$ is injective in category $\cO_{c}$. Indeed, this follows from the fact that $\Delta_{c}(\triv)$ is projective, which in turn follows from $\triv$ being minimal in the order giving $\cO_{c}$ a highest weight structure. Now, there is a duality functor  $\cO_{c} \to \cO_{c}$ sending $\Delta_{c}(\triv) \mapsto \nabla_{c}(\triv)$, cf. \cite[4.2]{GGOR}. From here, the injectivity of $\nabla_{c}(\triv)$ follows. 

Now, since $\nabla_{c}(\triv)$ is injective and the left adjoint to $\homf{\Delta_{c}(\triv)}{\bullet}$ is exact, $D_{c}$ is injective as a HC bimodule. Since $\nabla_{c}(\triv)$ is uniserial, so is $D_{c}$, and its socle coincides with $\homf{\Delta_{c}(\triv)}{L_{c}(\triv)}$. The bimodule constructed in \cite{BL} is also indecomposable and has the same socle, so it admits an embedding to $D_{c}$. But the composition length of $D_{c}$ is at most that of $\nabla_{c}(\triv)$, and this coincides with the composition length of the bimodule constructed in \cite{BL}. So the bimodules are isomorphic. 
\end{proof}

\begin{lemma}\label{lemma:injproj}
Both the double wall-crossing bimodule $D_{c}$ and the regular bimodule $H_{c}$ are injective-projective in the category $\HC(c,c)$.
\end{lemma}
\begin{proof}
That $D_{c}$ is injective was already observed in the proof of Lemma \ref{lemma:BL}. That $H_{c}$ is injective is \cite[Proposition 2.11]{hcbimod}. Now, there is a duality functor that sends the double wall-crossing bimodule to the regular bimodule and viceversa. This is constructed in Section \ref{sect:duality}, which is independent of the intervening material, see Proposition \ref{prop:dcwvsdreg}. It follows that both bimodules are projective, too.
\end{proof}

\begin{lemma}\label{lemma:calculations}
Let $\underline{W} \subseteq S_{n}$ be a parabolic subgroup. Let $\underline{D}_{c}, \underline{H}_{c}$ denote the double wall-crossing and the regular $H_{c}(\underline{W})$-modules, and let $\heartsuit := \heartsuit^{S_{n}}_{\underline{W}}$. Then,
\begin{enumerate}
\item $H_{c, \heartsuit} = \underline{H}_{c}$, $D_{c, \heartsuit} = \underline{D}_{c}$. 
\item $\underline{H}_{c}^{\heartsuit} = H_{c}$. 
\item $^{\heartsuit}\!\underline{D}_{c} = D_{c}$.
\end{enumerate}
\end{lemma}
\begin{proof}
Statement (1) follows because $\Res^{S_{n}}_{\underline{W}}(\Delta(\triv)) = \Delta(\triv_{\underline{W}})$ and a similar equation with $\Delta$ replaced by $\nabla$, cf. Corollary \ref{cor:respoly}. Let us move to (2). Since $\bullet_{\heartsuit}$ is exact and $\bullet^{\heartsuit}$ is its right adjoint, $\bullet^{\heartsuit}$ has to preserve injectives. Since $H_{c}$ is injective by Lemma \ref{lemma:injproj}, it follows that $H_{c}^{\heartsuit}$ is injective as well. Now, for simple $B \in \HC(c,c)$, $\Hom_{\HC}(B, \underline{H}_{c}^{\heartsuit}) = \Hom_{\underline{\HC}}(B_{\heartsuit}, \underline{H}_{c})$. But every sub-bimodule of the regular bimodule has full support. It follows that $\Hom_{\HC}(B, \underline{H}_{c}^{\heartsuit}) \neq 0$ if and only if $B$ has full support. So $\underline{H}_{c}^{\heartsuit}$ has to be a direct sum of copies of the injective hull of the unique simple with full support. But this is precisely the regular bimodule. So $\underline{H}_{c}^{\heartsuit} = H_{c}^{\oplus r}$ for some $r > 0$. Let us show that $r = 1$. For this, it is enough to check that $\Hom_{\underline{\HC}}(B_{\heartsuit}, \underline{H}_{c})$ is $1$-dimensional where, as before, $B$ is the unique simple sub-bimodule of $H_{c}$. Moreover, we claim that $\Hom_{\underline{\HC}}(X, \underline{H}_{c})$ is $1$-dimensional for \emph{every} sub-bimodule $X$ of $\underline{H}_{c}$. Indeed, since $\underline{H}_{c}$ is injective in $\underline{\HC}(c,c)$, every map $X \to \underline{H}_{c}$ can be extended to an automorphism of $\underline{H}_{c}$. But $\End_{\underline{\HC}}(\underline{H}_{c}, \underline{H}_{c})$ coincides with the center of $\underline{H}_{c}$, which is trivial. The claim follows. The proof of (3) is similar, now using the fact that $D_{c}$ is the projective cover of the unique simple with full support, and that we have $\End_{\HC}(\underline{D}_{c}) \cong \End_{\HC}(\underline{H}_{c})$, cf. Proposition \ref{prop:dcwvsdreg}, which is independent of the intervening material.
\end{proof}

Let us now take $\underline{W} = \{1\}$. So the algebra $\underline{H}_{c}$ is 1-dimensional, and $\underline{H}_{c} = \underline{D}_{c}$. Thus, in this case we have $^{\heartsuit}\!\underline{H}_{c} = D_{c}$, while $\underline{H}_{c}^{\heartsuit} = H_{c}$. This shows, in particular, that the induction and coinduction functors need not be isomorphic.

\begin{remark}\label{rmk:equivariantres}
Statement (1) of Lemma \ref{lemma:calculations} is clearly also valid when \lq\lq \,$\heartsuit$\rq\rq \, is replaced by \lq\lq\,$\dagger$\rq\rq. This shows, in particular, that $\underline{D}_{c}$ admits a $\Xi$-equivariant structure where, recall, $\Xi = N_{W}(\underline{W})/\underline{W}$. The statement (2) is also valid for the equivariant induction functor. Indeed, it is easy to show that $\underline{H}_{c}$ is injective also in the category of $\Xi$-equivariant bimodules, and the proof now goes as in that of Lemma \ref{lemma:calculations}.
\end{remark}

\subsection{Two-parametric case} We would like to say a few words about the category $\HC(c,c')$ where $c$ and $c'$ are different parameters. The following comprises the results in \cite[Section 6]{hcbimod}.

\begin{theorem}
The category $\HC(c,c')$ is nonzero if and only if either
\begin{enumerate}
\item $c - c' \in \ZZ$ or $c+c' \in \ZZ$; or
\item $c = r/m, c' = r'/m$ with $\gcd(r,m) = \gcd(r',m) = 1$ and $m$ divides $n$.
\end{enumerate}

In Case 1, we have an equivalence $D^{b}_{\HC}((H_{c}, H_{c'})\bimod) \cong D^{b}_{\HC}(H_{c}\bimod) \cong D^{b}_{\HC}(H_{c'}\bimod)$. If $(c,c')$ falls in Case 2 but not in Case 1, then $\HC(c,c')$ is equivalent to the category of representations of the group $S_{n/m}$. 
\end{theorem}

So the study of the category $\HC(c,c')$ can be reduced to a large extent to the case $c = c'$.




\section{Duality}\label{sect:duality} Now we come to the task of constructing a duality functor on categories of HC bimodules. We will do this in the generality of quantized quiver varieties, of which spherical rational Cherednik algebras of type $A$ are a special case. Other important cases of this construction are algebras of differential operators on flag varieties, spherical cyclotomic rational Cherednik algebras, and quantizations of Gieseker varieties, cf. \cite{ginzburg_lectures, EGGO, EKLS, gordon_rmk, losev_isoquant, losev_gieseker, oblomkov} 

\subsection{Quiver varieties}

In this section, we recall the definition and some properties of Nakajima quiver varieties. These are defined as the GIT Hamiltonian reduction of  a reductive group acting on a representation space of a quiver. Thus, first we study this action and the respective moment map. Later, we study the GIT properties of the action, giving in particular a concrete description of the semistable points (for a particular choice of stability condition). After that, we define both the affine and projective Nakajima quiver varieties, as well as their universal deformations. This section does not contain new results. Nakajima quiver varieties were first studied in \cite{nakajima}, and our exposition here mostly follows \cite{nakajima_lectures, ginzburg_lectures}.

\subsubsection{Representation spaces and moment maps.} Let $Q = (Q_0, Q_1, s, t)$ be a quiver: $Q_0$ is the set of vertices; $Q_1$ the set of arrows; and $s, t: Q_{1} \rightarrow Q_{0}$ the maps that to each arrow assign its starting and terminating vertex, respectively. We will consider the \emph{co-framed} quiver:
\begin{equation*}
Q^{\heartsuit} := (Q_{0}^{\heartsuit} = Q_{0} \sqcup Q_{0}', Q_{1}^{\heartsuit} := Q_{1} \sqcup Q_{1}', s^{\heartsuit}, t^{\heartsuit})
\end{equation*}

\noindent where $Q_{0}' := \{k' : k \in Q_{0}\}$ is a copy of $Q_{0}$; $Q_{1}' := \{\alpha_{i} : i \in Q_{0}\}$ is a set of arrows indexed by $Q_{0}$; $s^{\heartsuit}|_{Q_{1}} = s$, $s(\alpha_{k}) = k$, $t^{\heartsuit}|_{Q_{1}} = t$, and $t^{\heartsuit}(\alpha_{k}) = k'$. In pedestrian terms, the quiver $Q^{\heartsuit}$ is obtained from $Q$ by attaching to each vertex $k \in Q_0$ a coframing vertex $k'$ with an arrow $k \rightarrow k'$. 

Now let $\bv, \bw \in \ZZ^{Q_0}_{\geq 0}$ be dimension vectors. We can consider the space of representations:

\begin{equation*}
R(\bv, \bw) := \Rep(Q^{\heartsuit}, (\bv, \bw)) = \bigoplus_{\alpha \in Q_{1}} \Mat(\bv_{s(\alpha)}, \bv_{t(\alpha)}) \oplus \bigoplus_{k \in Q_{0}} \Mat(\bv_{k}, \bw_{k})
\end{equation*}

We will denote an element of $R(\bv, \bw)$ by $(X_{\alpha}, i_{k})_{\alpha \in Q_{1}, k \in Q_{0}}$.Note that the reductive group $G := \GL(\bv) := \prod_{k \in Q_{0}} \GL(\bv_{k})$ acts on the space $R(\bv, \bw)$ by changing bases. For $g = (g_{k})_{k \in Q_{0}} \in G$ we get

\begin{equation*}
g.(X_{\alpha}, i_{k}) = (g_{t(\alpha)}X_{\alpha}g_{s(\alpha)}^{-1}, i_{k}g_{k}^{-1})
\end{equation*} 

We take the induced action of $G$ on $T^{*}R$. Since this is an induced action, it is Hamiltonian, let us describe the moment map. First, we will describe this action in linear-algebraic terms. For any $m > 0$ we have a $\GL(n)$-equivariant isomorphism $\Mat(n, m)^{*} \cong \Mat(m,n)$ which is given by the trace form. Thus, we have a $G$-equivariant identification of $T^{*}\!R$ with the space of representations of the double quiver $\overline{Q^{\heartsuit}}$
\begin{equation*}
T^{*}\!R = \bigoplus_{\alpha \in Q_{1}}\left(\Mat(\bv_{s(\alpha)}, \bv_{t(\alpha)}) \oplus \Mat(\bv_{t(\alpha)}, \bv_{s(\alpha)})\right) \oplus \bigoplus_{k \in Q_{0}}\left(\Mat(\bv_{k}, \bw_{k}) \oplus \Mat(\bw_{k}, \bv_{k})\right)
\end{equation*}

We will denote an element of $T^{*}\!R$ by $(X_{\alpha}, Y_{\alpha}, i_{k}, j_{k})$. Thus, the action of $G$ is given by 
\begin{equation*}
g.(X_{\alpha}, Y_{\alpha}, i_{k}, j_{k}) = (g_{t(\alpha)}X_{\alpha}g_{s(\alpha)}^{-1}, g_{s(\alpha)}Y_{\alpha}g_{t(\alpha)}^{-1}, i_{k}g_{k}^{-1}, g_{k}j_{k})
\end{equation*}

And the moment map is given by
\begin{equation}\label{eqn:momentmap}
\mu: T^{*}\!R \rightarrow \gl(\bv)^{*} = \gl(\bv) =: \fg, \hspace{0.5cm} (X_{\alpha}, Y_{\alpha}, i_{k}, j_{k}) \mapsto \sum_{\alpha \in Q_{1}} (Y_{\alpha}X_{\alpha} - X_{\alpha}Y_{\alpha}) - \sum_{k \in Q_{0}} j_{k}i_{k}
\end{equation}

\noindent where we have used the trace form to identify $\gl(\bv)$ with its dual.  The dual to this map is the \emph{comoment map} $\mu^{*}: \fg \rightarrow \CC[T^{*}\!R]$. Since the moment map $\mu$ is $G$-equivariant, for every $\lambda \in (\fg^{*})^{G}$ the group $G$ acts on $\mu^{-1}(\lambda)$. For $\lambda \in (\fg^{*})^{G}$, we will be interested in the affine variety
\[
\cM_{\lambda}^0 := \cM_{\lambda}^0(\bv, \bw) := \mu^{-1}(\lambda)/\!/G = \Spec(\CC[\mu^{-1}(\lambda)]^{G}) = \Spec([\CC[T^{*}\!R]/(\{\mu^{*}(\xi) - \langle \lambda, \xi\rangle : \xi \in \fg\})]^{G})
\]

Note that the variety $\cM^{0}_{0}$ comes equipped with a $\CC^{\times}$-action induced by the $\CC^{\times}$ action on $T^{*}R$ by dilations. In particular, the Poisson bracket on $\cM^{0}_{0}$ has degree $-2$. In general, $\cM^{0}_{\lambda}$ is singular. In some cases, we can construct resolutions of singularities by looking at GIT quotients of $T^{*}R$. This is what we will do next.

\subsubsection{G.I.T. quotients}

Note that the group of characters $\fX(\GL(\bv))$ may be identified with $\ZZ^{Q_{0}}$, to $\theta = (\theta_{k})_{k \in Q_{0}}$ we associate the character $(g_{k}) \mapsto \prod_{k \in Q_{0}} \det(g_{k})^{\theta_{k}}$. For $\theta \in Q_{0}$, we have the graded algebra of semi-invariants
\begin{equation}\label{eqn:GIT}
\CC[T^{*}R]^{\theta-si} := \bigoplus_{n \geq 0} \CC[T^{*}R]^{n\theta}
\end{equation}

\noindent where $\CC[T^{*}R]^{n\theta} := \{f \in \CC[T^{*}R] : f(g^{-1}x) = \theta^{n}(g)f(x) \; \text{for every} \; x \in T^{*}R, g \in \GL(\bv)\}$. 

%


Recall that a point $x \in T^{*}R$ is said to be $\theta$-semistable if there exist $n > 0$ and $f \in \CC[T^{*}R]^{n\theta}$ such that $f(x) \neq 0$. Let us describe the semistable points with respect to certain stability conditions. First of all, for a $\overline{Q^{\heartsuit}}$-representation $x = (X_{\alpha}, Y_{\alpha}, i_{k}, j_{k}) \in T^{*}\!R = Rep(\overline{Q^{\heartsuit}}, (\bv, \bw))$ we denote by $\underline{x} = (X_{\alpha}, Y_{\alpha}) \in Rep(\overline{Q}, (\bv))$ the representation of $\overline{Q}$ that is obtained by forgetting the framing and coframing maps. 

\begin{proposition}[Proposition 5.1.5, \cite{ginzburg_lectures}]
\begin{enumerate}
\item Assume $\theta \in \ZZ_{>0}^{Q_{0}}$. Then, a representation $x = (X_{\alpha}, Y_{\alpha}, i_{k}, j_{k}) \in T^{*}R$ is $\theta$-semistable if and only if the only subrepresentation of $\underline{x}$ which contains $(\im(j_{k}))_{k \in Q_{0}}$ is the entire representation $\underline{x}$.
\item Assume $\theta \in \ZZ_{<0}^{Q_{0}}$. Then, a representation $x = (X_{\alpha}, Y_{\alpha}, i_{k}, j_{k}) \in T^{*}R$ is $\theta$-semistable if and only if the only subrepresentation of $\underline{x}$ contained in $(\ker(i_{k}))_{k \in Q_{0}}$ is the $0$ representation.
\end{enumerate}
\end{proposition}

We will denote $\theta^{+} := (1, 1, \dots, 1)$ and $\theta^{-} := -\theta^{+}$. In particular, $\theta^{+}$ falls under (1) of the previous proposition, while $\theta^{-}$ falls under (2). 

Now we proceed to define the GIT Hamiltonian reduction of $T^{*}R$ by the action of $G$. First of all, obviously $G$ acts on $\mu^{-1}(\lambda)$ for $\lambda \in (\fg^{*})^{G}$. So, similarly to (\ref{eqn:GIT}) we may define the algebra of semi-invariants $\CC[\mu^{-1}(\lambda)]^{\theta-si}$, and the set of $\theta$-semistable points $\mu^{-1}(\lambda)^{\theta-ss}$. We remark that $\mu^{-1}(\lambda)^{\theta-ss} = (T^{*}R)^{\theta-ss}\cap\mu^{-1}(\lambda)$, this is a consequence of the Hilbert-Mumford criterion. Now we can define the GIT Hamiltonian reduction
\begin{equation}
\cM^{\theta}_{\lambda} := \Proj(\CC[\mu^{-1}(\lambda)]^{\theta-si}) = \mu^{-1}(\lambda)^{\theta-ss}/\!/ G
\end{equation}

Let us remark, first, that the formalism of Hamiltonian reduction implies that the symplectic structure on $T^{*}R$ gives $\cM^{\theta}_{\lambda}$ the structure of an algebraic Poisson variety. We also remark that the $0$th graded component of the algebra $\CC[\mu^{-1}(\lambda)]^{\theta-si}$ is precisely $\CC[\mu^{-1}(\lambda)]^{G}$, so we have a projective morphism $\varpi: \cM^{\theta}_{\lambda} \rightarrow \cM^{0}_{\lambda}$. We will state sufficient conditions for this morphism to be a resolution of singularities.

 Of course, for $\varpi: \cM^{\theta}_{\lambda} \rightarrow \cM^{0}_{\lambda}$ to be a resolution of singularities we need, first, that $\cM^{\theta}_{\lambda}$ is smooth. The variety $\cM^{\theta}_{\lambda}$ will be smooth and symplectic provided the $G$-action on $\mu^{-1}(\lambda)^{\theta-ss}$ is free. When this happens, we will say that the pair $(\theta, \lambda)$ is \emph{generic}. The following result gives a sufficient condition for a pair $(\theta, \lambda)$ to be generic. We denote by $\fg(Q)$ the Kac-Moody algebra associated to the quiver $Q$. By $\alpha^{i}$, we mean the simple root of $\fg(Q)$ corresponding to the vertex $i \in Q_{0}$.
 
 \begin{proposition}
 	The pair $(\theta, \lambda)$ is generic if there is no $\bv' \in \ZZ_{\geq 0}^{Q_{0}}$ such that:
 	\begin{enumerate}
 		\item Componentwise, $\bv' \leq \bv$.
 		\item $\sum_{i \in Q_{0}} \bv'_{i}\alpha^{i}$ is a root of $\fg(Q)$.
 		\item $\sum_{i \in Q_{0}} \bv'_{i}\lambda_{i} = \sum_{i \in Q_{0}} \bv'_{i}\theta_{i} = 0$.
 	\end{enumerate}
 \end{proposition} 

For example, both $\theta^{+}$ and $\theta^{-}$ are generic, in the sense that $(\theta^{\pm}, \lambda)$ is generic for \emph{any} $\lambda \in (\fg^{*})^{G}$. This will be important for us later.

\emph{\underline{From now on, we assume that the pair $(\theta, \lambda)$ is generic}}. Let us denote by $\cM_{\lambda}$ the affine variety $\Spec[\Gamma(\cM^{\theta}_{\lambda}, \fS_{\cM^{\theta}_{\lambda}})]$ \footnote{Following the conventions of \cite{BPW, BLPW}, to avoid confusion with category $\cO$, for a variety $X$ we will denote its structure sheaf by $\fS_{X}$.}. In particular, we have a projective morphism $\rho: \cM^{\theta}_{\lambda} \rightarrow \cM_{\lambda}$.

\begin{proposition}[Proposition 2.3, \cite{BL}]
	The map $\rho$ is a resolution of singularities. Moreover, the variety $\cM_{\lambda}$ is independent of the stability condition $\theta$.
\end{proposition}

So the map $\varpi$ will be a resolution of singularities provided we have $\cM_{\lambda} = \cM^{0}_{\lambda}$ (assuming $(\theta, \lambda)$  is generic). This is the case when the moment map $\mu$ is flat, \cite[Proposition 2.5]{BL}. Sufficient conditions for this to happen were found by Crawley-Boevey in \cite{CB} see, for example, Theorem 1.1 in \emph{loc. cit.} A consequence of this is that $\mu$ is flat whenever $Q$ is a finite or affine quiver and $\nu := \sum_{i \in Q_{0}}(\bw_{i}\omega_{i} - \bv_{i}\alpha_{i})$ is a dominant weight for $\fg(Q)$, where $\omega_{i}, \alpha_{i}$ denote the fundamental weights and simple roots for $\fg(Q)$, respectively.  In general, all constructions that follow remain valid if we replace the variety $\cM^{0}_{\lambda}$ by $\cM_{\lambda}$. 

\subsubsection{Universal quiver varieties}
Now assume that the stability condition $\theta$ is such that $(\theta, \lambda)$ is a generic pair for every $\lambda$. For example, we can take $\theta = \theta^{+}$ or $\theta^{-}$. In this case, the action of $G$ on $\mu^{-1}((\fg^{*})^{G})^{\theta-ss}$ is free, and so the universal quiver variety
\[
\cM^{\theta}_{\fp} := \mu^{-1}((\fg^{*})^{G})^{\theta-ss}/\!/ G
\]

\noindent is smooth and symplectic relative to $\mathfrak{p}$. We will also denote
\[
 \cM^{0}_{\fp} := \mu^{-1}((\fg^{*})^{G})/\!/G, \; \; \; \cM_{\fp} := \Spec(\Gamma(\cM^{\theta}_{\fp}, \fS_{\cM^{\theta}_{\fp}})).
\]

We remark that $\cM^{\theta}_{\fp}$ is a scheme over $\fp := (\fg^{*})^{G}$ and its specialization to $\lambda \in \fp$ coincides with $\cM^{\theta}_{\lambda}$. In particular, $\cM_{\fp}^{\theta}$ is a deformation of $\cM^{\theta}_{0}$ over $\fp$.  Note that we have an action of $\CC^{\times}$ on $\cM^{\theta}_{\fp}$ that restricts to the usual action on  the fiber over $0$ $\cM^{\theta}_{0}$. On the other hand, Namikawa, cf. \cite{namikawa}, has proved that the variety $\cM^{\theta}_{0}$ admits a universal deformation $\mathscr{M}^{\theta}$ over the space $H^{2}_{\DR}(\cM^{\theta}_{0}, \CC)$. Let us see the relation between the universal deformation $\mathscr{M}^{\theta}$ and the universal quiver variety $\cM^{\theta}_{\fp}$. We have a natural map $\fp \rightarrow H^{2}_{\DR}(\cM^{\theta}_{0}, \CC)$ given as follows. Let $\chi \in \fX(G)$ be a character. Consider the line bundle $V_{\chi}$ on $T^{*}R$, which is trivial as a line bundle and with a $G$-action given by $\chi^{-1}$. Since $V_{\chi}$ is $G$-equivariant, its restriction $V_{\chi}|_{\mu^{-1}(0)}^{\theta-ss}$ descends to a line bundle $\fS(\chi)$ on $\cM^{\theta}_{0}$. For example, the line bundle $\fS(\theta)$ is ample by definition. This defines a map $\iota: \fX(G) \rightarrow H^{2}_{\DR}(\cM_{0}^{\theta}, \CC)$, $\chi \mapsto c_{1}(\fS(\chi))$ (the first Chern class) which extends by linearity to $\iota: \fp \rightarrow H^{2}_{\DR}(\cM^{\theta}_{0}, \CC)$. 

\begin{theorem}[McGerty and Nevins, \cite{mcgertynevins}]\label{ass:universal}
Assume that $\bw \neq 0$. Then the map $\iota: \fp \rightarrow H^{2}_{\DR}(\cM^{\theta}_{0}, \CC)$ is surjective.  
\end{theorem}

\emph{\underline{From now on, we assume that the coframing vector $\bw$ is nonzero.}} \\

Now, $\cM^{\theta}_{\fp}$ is a deformation of $\cM^{\theta}_{0}$ over $\fp$, while $\scrM^{\theta}$ is a deformation of $\cM^{\theta}_{0}$ over $H^{2}_{\DR}(\cM^{\theta}_{0}, \CC)$. It follows from the universality of $\scrM^{\theta}$ that we have
\[
\iota^{*}\scrM^{\theta} \cong \cM^{\theta}_{\fp}.
\]


\subsubsection{Isomorphisms between quiver varieties}\label{subsect:iso_6}
Let us remark that we have the following $\CC^{\times}$-equivariant symplectomorphism of $T^{*}R$
\begin{equation}\label{eqn:endo}
\begin{array}{rl}
\Upsilon: T^{*}R \rightarrow & T^{*}R \\
(X_{\alpha}, Y_{\alpha}, i_{k}, j_{k}) \mapsto & (-Y_{\alpha}^{t}, X_{\alpha}^{t}, -j_{k}^{t}, i_{k}^{t})
\end{array}
\end{equation}

\noindent where $\bullet^{t}$ denotes matrix transposition. Note, however, that this map is \emph{not} $G$-equivariant. Rather, we have that $\Upsilon(g.x) = (g^{t})^{-1}.\Upsilon(x)$. This implies the following.

\begin{lemma}
\begin{enumerate}
\item For any character $\theta \in \fX(G)$, the map $\Upsilon$ induces a graded isomorphism $\Upsilon^{*}: \CC[T^{*}R]^{\theta-si} \rightarrow \CC[T^{*}R]^{-\theta-si}$.
\item The following diagram commutes
\[
\xymatrix{T^{*}R \ar[rr]^{\Upsilon} \ar[d]^{\mu} & & T^{*}R\ar[d]_{\mu} \\
\fg \ar[rr]^{\xi \mapsto -\xi^{t}} & & \fg}
\]
\end{enumerate}
\end{lemma}

\begin{corollary}
For any $\lambda \in (\fg^{*})^{G}$, we have an induced, graded isomorphism
\[
\CC[\mu^{-1}(\lambda)]^{\theta-si} \rightarrow \CC[\mu^{-1}(-\lambda)]^{-\theta-si}
\]

\noindent and consequently we have an isomorphism of projective varieties $\cM^{\theta}_{\lambda} \rightarrow \cM^{-\theta}_{-\lambda}$. These isomorphisms glue together to an isomorphism $\cM^{\theta}_{\fp} \rightarrow \cM^{-\theta}_{\fp}$ such that the following diagram commutes
\[
\xymatrix{ \cM^{\theta}_{\fp} \ar[rr] \ar[d] & & \cM^{-\theta}_{\fp}\ar[d] \\ \fp \ar[rr]^{\lambda \mapsto - \lambda} & & \fp}
\]
\end{corollary}

Let us remark that $\Upsilon$ also induces a $\CC^{\times}$-equivariant Poisson automorphism of the affine quiver variety $\Upsilon: \cM^{0}_{0} \rightarrow \cM^{0}_{0}$. We will denote by $\Upsilon^{*}: \CC[\cM^{0}_{0}] \rightarrow \CC[\cM^{0}_{0}]$ the induced automorphism on its algebra of functions.

\subsection{Quantizations}
Let us proceed to quantizations of the quiver varieties $\cM_{0}^{0}, \cM_{0}^{\theta}$ and $\cM^{\theta}_{\fp}$. Since $\cM^{0}_{0}$ is an affine variety, its quantizations are straightforward to define. It is slightly harder to define quantizations of the varieties $\cM_{0}^{\theta}$ and $\cM_{\fp}^{\theta}$. We follow, mostly, \cite{BL}. The only new result of this section is in Section \ref{subsect:isoquant}, see (\ref{eqn:iso_important}), but it is not really original.

\subsubsection{Quantizations of $\cM^{0}_{0}$}
Recall that $\cM^{0}_{0}$ is an affine, Poisson variety with a $\CC^{\times}$-action. Moreover, note that the Poisson bracket has degree $-2$ with respect to the $\CC^{\times}$-action. 

\begin{definition}
A \emph{quantization} of $\cM^{0}_{0}$ is a pair $(\cA, \iota)$ where

\begin{enumerate}
\item $\cA$ is an associative, filtered algebra $\cA = \bigcup_{i \geq 0} \cA^{i}$ such that, for $a \in \cA^{i}, b \in \cA^{j}$, $[a,b] \in \cA^{i + j - 2}$ (in particular, this implies that $\gr\cA$ is equipped with a Poisson bracket of degree $-2$.)
\item $\iota: \gr\cA \rightarrow \CC[\cM^{0}_{0}]$ is an isomorphism of graded Poisson algebras.
\end{enumerate}
\end{definition}

By an \emph{isomorphism} of quantizations we mean a filtered isomorphism $f: \cA \rightarrow \cA'$ that induces the identity on $\CC[\cM^{0}_{0}]$ on the associated graded level. Quantizations of the variety $\cM^{0}_{0}$  can be obtained using \emph{quantum} Hamiltonian reduction, as follows.

Let $G$ be a connected reductive algebraic group acting on an algebra $\cA$ by algebra automorphisms. In particular, the Lie algebra $\fg$ acts on $\cA$ by derivations. For $\xi \in \fg$, let us denote by $\xi_{\cA}: \cA \rightarrow \cA$ the corresponding derivation. We say that a map $\Phi: \fg \rightarrow \cA$ is a \emph{quantum comoment map} if it is $G$-equivariant and, for $\xi \in \fg$, $\xi_{\cA} = [\Phi(\fg), \cdot]$. Note, in particular, that $\Phi([\xi, \eta]) = [\Phi(\xi), \Phi(\eta)]$. So the quantum comoment map extends to an algebra map $\Phi: \cU(\fg) \rightarrow \cA$.

For a character $\lambda \in (\fg^{*})^{G}$, we consider the ideal $\mathcal{I}_{\lambda} := \cU(\fg)\{\xi - \langle \lambda, \xi\rangle : \xi \in \fg\}$. Now we define the \emph{quantum Hamiltonian reduction}
\[
\cA_{\lambda} := \cA/\!/\!/\!_{\lambda}G :=  \left[\cA/\cA\Phi(\cI_{\lambda})\right]^{G}.
\]

We remark that $\cA_{\lambda}$ has an algebra structure induced from the algebra structure on $\cA$. Moreover, if we denote by $M_{\lambda}$ the cyclic $\cA$-module $\cA/\cA\Phi(\cI_{\lambda})$ then $\cA_{\lambda} = \End_{\cA}(M_{\lambda})^{\opp}$. As in the classical case, we may define a \lq\lq universal\rq\rq \; quantum Hamiltonian reduction, as follows. Denote $\fP := \fg/[\fg, \fg] (\cong (\fg^{*})^{G}$). Consider the ideal $\cI := \cU(\fg)[\fg,\fg] \subseteq \cU(\fg)$. The universal quantum Hamiltonian reduction is
\[
\cA_{\fP} := \left[\cA/\cA\Phi(\mathcal{I})\right]^{G}.
\]

We remark that this is an $S(\fP)$-algebra, that is, the image of $\fP$ under the natural map $\fP \rightarrow \cA_{\fP}$ is contained in the center of $\cA_{\fP}$. Under the natural identification $\fP \cong (\fg^{*})^{G}$, we have that $\cA_{\lambda}$ coincides with the specialization of $\cA_{\fP}$ at $\lambda$. 

Let us now specialize to the case of Nakajima quiver varieties. Here, the algebra $\cA$ is the algebra of global differential operators on $R$, $D(R)$. The quantum comoment map is given as follows: the action of $G$ on $R$ induces a map $\fg \rightarrow \Vect_{R}$, $\xi \mapsto \xi_{R}$. But $\Vect_{R} \subseteq D(R)$, so we may consider this as a map to $D(R)$, and this is a quantum comoment map. For $\lambda \in (\fg^{*})^{G}$, we will denote by $\widehat{\cA}_{\lambda}$ the quantum Hamiltonian reduction of $D(R)$ at $\lambda$. If the moment map $\mu$ is flat, this is a quantization of $\cM^{0}_{0}$.

\subsubsection{Quantizations of $\cM^{\theta}_{0}$}\label{sect:sheafquant}

Now we want to define quantizations of the variety $\cM^{\theta}_{0}$. Since this is not an affine variety, a quantization of it will not be specified by a single algebra of functions. Rather, we need a sheaf on $\cM^{\theta}_{0}$ whose associated graded coincides with the structure sheaf $\fS_{\cM^{\theta}_{0}}$. Note, however, that the last sentence does not make sense as stated: for an open set $U \subseteq \cM^{\theta}_{0}$, the algebra $\Gamma(U, \fS_{\cM^{\theta}_{0}})$ is only naturally graded when $U$ is $\CC^{\times}$-stable. So, before, we need to change the topology.

\begin{definition}
The \emph{conical topology} on $\cM^{\theta}_{0}$ is that topology whose open set are the Zariski open $\CC^{\times}$-stable subsets of $\cM^{\theta}_{0}$. 
\end{definition} 

We remark that every point $x \in \cM^{\theta}_{0}$ has a Zariski open neighborhood which is affine and $\CC^{\times}$-stable, this is known as Sumihiro's theorem, \cite{sumihiro}. So the conical topology is still fine enough for most purposes. A quantization of $\cM^{\theta}_{0}$ will then be a filtered sheaf of algebras in the conical topology whose associated graded coincides with $\fS_{\cM^{\theta}_{0}}$. For technical reasons, this sheaf of algebras is supposed to satisfy some conditions which we now make precise.

\begin{definition}\label{defn:sheafquant}
A \emph{quantization} of $\cM^{\theta}_{0}$ is a pair $(\cA^{\theta}, \iota)$ where $\cA^{\theta}$ is a filtered sheaf of algebras in the conical topology and $\iota: \gr\cA^{\theta} \rightarrow \fS_{\cM^{\theta}_{0}}$ is a Poisson isomorphism, such that the filtration on $\cA^{\theta}$ satisfies the following conditions.

\begin{enumerate}
\item It is complete, meaning that every Cauchy sequence in the topology determined by the filtration converges. More explicitly, we require that for every sequence $\{x_{i}\}_{i = 1}^{\infty}$ of sections of $\cA^{\theta}$ such that for every $n \in \ZZ$ there exist $i_{0}$ such that $x_{i} - x_{j} \in (\cA^{\theta})^{\leq n}$ for $i, j > i_{0}$, there exists a section $x \in \cA^{\theta}$ such that for every $m \in \ZZ$ there exist $j_{0}$ such that $x - x_{i} \in (\cA^{\theta})^{\leq m}$ for every $i > j_{0}$.
\item It is separated, meaning that $\bigcap_{i \in \ZZ}(\cA^{\theta})^{\leq i} = 0$. Equivalently, the limit $x$ of the previous paragraph is unique. 
\end{enumerate}
\end{definition}

Quantizations of $\cM^{\theta}_{0}$ may be produced similarly to those of $\cM^{0}_{0}$. Instead of using the algebra $D(R)$, we use the \emph{microlocalization} of the sheaf of differential operators on $R$. This is a sheaf $D_{R}$ on the conical topology of $T^{*}R$, where the action of $\CC^{\times}$ is by dilations on the cotangent fibers, see, for example, \cite{ginzburg_char}, \cite{kashiwara} for details on microlocalization. The global sections of $D_{R}$ coincide with $D(R)$. 

Now let $f \in \CC[T^{*}R]$ be an $n\theta$-semiinvariant element, where $n > 0$. Then, we have the open set $X_{f} \subseteq \cM^{\theta}_{0}$ that is $\mu^{-1}(0)\cap\{f \neq 0\}/\!/ G$. That the open sets $X_{f}$ form a base of the topology of $\cM^{\theta}_{0}$ follows by the definition of $\Proj$. It is easy to see that if, moreover, $f$ is homogeneous with respect to the $\CC^{\times}$-action on $T^{*}R$, then $X_{f}$ is open in the conical topology, and the sets $X_{f}$ form a basis for the conical topology of $\cM^{\theta}_{0}$.

So, for $\lambda \in (\fg^{*})^{G}$ we define the sheaf $\cD^{\theta}_{\lambda}$ by setting, on an open set $X_{f}$:
\[
\cD^{\theta}_{\lambda}(X_{f}) := D_{R}((T^{*}R)_{f})/\!/\!/\!_{\lambda}G.
\]

It is possible to see that $\cD^{\theta}_{\lambda}$ defined in this way is a quantization of $\cM^{\theta}$. Let us denote by $\cD_{\lambda}$ the algebra of global sections of $\cD^{\theta}_{\lambda}$. This is a quantization of $\cM_{0}$. If the moment map $\mu$ is flat, then we actually have that $\cD_{\lambda} = \widehat{\cA}_{\lambda}$. For proofs of these statements, see \cite[Section 2]{BL}.

Let us remark that we also have the notion of quantizations of the universal quiver variety $\cM^{\theta}_{\fP}$. This is a sheaf of $\CC[\fp]$-algebras in the conical topology (recall that we have an action of $\CC^{\times}$ on $\cM^{\theta}_{\fp}$) satisfying conditions analogous to those of Definition \ref{defn:sheafquant}, see \cite[Section 3]{BPW}.

Isomorphism classes of quantizations of $\cM^{\theta}_{0}$ and $\cM^{\theta}_{\fp}$ have been parametrized in \cite{BK}, \cite{losev_isoquant} by the vector spaces $H^{2}_{DR}(\cM^{\theta}_{0}, \CC)$ and $H^{2}_{DR}(\cM^{\theta}_{\fp}/\fp, \CC)$, respectively. For $\lambda \in H^{2}_{DR}(\cM^{\theta}_{0}, \CC) = \fp$, we denote by $\cA^{\theta}_{\lambda}$ the quantization of $\cM^{\theta}_{0}$ with period $\lambda$. We remark that it is \emph{not} the case that $\cA^{\theta}_{\lambda}$ coincides with $\cD^{\theta}_{\lambda}$, for this we would have to take a symmetrized quantum comoment map, see \cite[Section 3.4]{BPW}, \cite[Section 3.2]{losev_isoquant}.

The quantization of $\cM^{\theta}_{0}$ (or of $\cM^{\theta}_{\fp}$) with period $0$ is called the \emph{canonical} quantization. It is characterized by the fact that it is isomorphic, as a quantization, to its opposite, this is \cite[Corollary 2.3.3]{losev_isoquant}. In fact, for $\lambda \in H^{2}_{DR}(\cM^{\theta}_{0}, \CC)$ we have $\cA_{-\lambda}^{\theta} = (\cA^{\theta}_{\lambda})^{\opp}$, cf. \cite{losev_isoquant}. 

\subsubsection{Isomorphisms of quantizations}\label{subsect:isoquant}
Now let $\scrA^{\theta}$ be the canonical quantization of the universal deformation $\scrM^{\theta}$ of $\cM^{\theta}_{0}$. We will consider its pullback $\iota^{*}\scrA^{\theta}$ under the Kirwan map. Since the canonical quantization is characterized by being isomorphic to its opposite, we have that $\iota^{*}\scrA^{\theta}$ is the canonical quantization of $\cM^{\theta}_{\fp}$.  This implies, in particular, that we have an isomorphism $\Upsilon^{*}\iota^{*}\scrA^{-\theta} \cong \iota^{*}\scrA^{\theta}$, where $\Upsilon: \cM^{\theta}_{\fp} \rightarrow \cM^{-\theta}_{\fp}$ is the isomorphism introduced in Subsection \ref{subsect:iso_6}. Thus, we have an induced isomorphism
\begin{equation}\label{eqn:iso_canonicalquant}
\Gamma(\scrM^{\theta}, \scrA^{\theta}) \buildrel \cong \over \rightarrow \Gamma(\scrM^{-\theta}, \scrA^{-\theta})
\end{equation}
where we have used that, since the map $\iota$ is surjective, $\Gamma(\scrM^{\theta}, \scrA^{\theta}) \cong \Gamma(\iota^{*}\scrM^{\theta}, \iota^{*}\scrA^{\theta}) = \Gamma(\cM^{\theta}_{\fp}, \iota^{*}\scrA^{\theta})$. We remark, however, that (\ref{eqn:iso_canonicalquant}) is \emph{not} an isomorphism of $\CC[\fp]$-algebras. Rather, it induces the automorphism on $\CC[\fp]$ given by $f(\lambda) \mapsto f(-\lambda)$. 

Now let $\lambda \in \fp$ be a period of quantization. According to \cite{BLPW}, we have an isomorphism of quantizations of $\cM^{0}_{0}$:
\begin{equation}\label{eqn:iso_quotients}
\cA_{\lambda} = \Gamma(\cM^{\theta}_{0}, \cA^{\theta}_{\lambda}) \cong \Gamma(\scrM^{\theta}, \scrA^{\theta})/\scrI_{\lambda}
\end{equation}

\noindent where $\scrI_{\lambda}$ is the ideal generated by the maximal ideal of $\lambda$, $\mathfrak{m}_{\lambda} \subseteq \CC[\fp] \subseteq \Gamma(\scrM^{\theta}, \scrA^{\theta})$. It follows from (\ref{eqn:iso_canonicalquant}) and (\ref{eqn:iso_quotients}) that we have an isomorphism 
\begin{equation}\label{eqn:iso_important}
\Phi_{\lambda}: \cA_{\lambda} \buildrel\cong\over\rightarrow \cA_{-\lambda}
\end{equation}

This is a filtered isomorphism that is, however, \emph{not} an isomorphism of quantizations. Indeed, it follows by construction that the associated graded of (\ref{eqn:iso_important}) coincides with the isomorphism $\Upsilon^{*}: \CC[\cM^{0}_{0}] \rightarrow \CC[\cM^{0}_{0}]$ constructed in Subsection \ref{subsect:iso_6}.

\subsection{Harish-Chandra bimodules}\label{sect:HCquant}

In this section we proceed to define Harish-Chandra bimodules. The definition is, basically, the same as with rational Cherednik algebras. However, for technical reasons (see (\ref{eqn:iso_important})) we define a wider class of bimodules, which we call \emph{twisted} HC bimodules. The twist here is provided by a $\CC^{\times}$-equivariant automorphism of $\cM^{0}_{0}$. When this automorphism is the identity, we recover the usual definition of HC bimodules. 

\subsubsection{Harish-Chandra bimodules: algebra level}

\begin{definition}
Let $\cA, \cA'$ be filtered quantizations of the same graded Poisson algebra $A$, and let $f: A \rightarrow A$ be a graded Poisson automorphism of $A$. We say that a $(\cA, \cA')$-bimodule $\cB$ is \emph{$f$-twisted Harish-Chandra} (shortly, \emph{$f$-HC}) if it admits a bimodule filtration such that:
\begin{enumerate}
\item $\gr\cB$ is a finitely generated $A$-bimodule.
\item For any $a \in A$, $b \in \gr\cB$, $ab = bf(a)$. 
\end{enumerate}

We denote the category of $f$-HC $(\cA, \cA')$-bimodules by $^{f}\!\HC(\cA, \cA')$. 
\end{definition}

We remark that, when $f$ is the identity, we recover the usual category of HC bimodules that is considered in \cite{BL, BLPW}. We will abbreviate $\HC(\cA, \cA') := \,  ^{\id}\!\HC(\cA, \cA')$, and call these bimodules simply HC. The following proposition is clear.

\begin{proposition}
\begin{enumerate}
\item Every $f$-HC $(\cA, \cA')$-bimodule $\cB$ is finitely generated as a left $\cA$-module and as a right $\cA'$-module.
\item The tensor product $\otimes_{\cA'}$ gives a bifunctor:
\[
^{g}\!\HC(\cA, \cA') \times \, ^{f}\!\HC(\cA', \cA'') \rightarrow \, ^{g \circ f}\!\HC(\cA, \cA'')
\]
\end{enumerate}
\end{proposition}

Let us specialize to the case where $\cA = \cA_{\lambda}$ is a quantization of the Nakajima quiver variety $\cM^{0}_{0}$. In this case, we write $^{f}\!\HC(\lambda, \lambda') := \,^{f}\!\HC(\cA_{\lambda}, \cA_{\lambda'})$. The following result is clear.

\begin{proposition}\label{prop:twist}
Twisting the right action by the isomorphism $\Phi_{-\mu}$ (see (\ref{eqn:iso_important})) gives a category equivalence $^{f}\!\HC(\lambda, \mu) \buildrel \cong \over \rightarrow \, ^{\Upsilon^{*} \circ f}\!\HC(\lambda, -\mu)$. Similarly, twisting the left action by the isomorphism $\Phi_{\lambda}$ gives a category equivalence $^{f}\!\HC(\lambda, \mu) \buildrel \cong \over \rightarrow \, ^{f \circ \Upsilon^{*}}\!\HC(-\lambda, \mu)$.
\end{proposition}

\subsubsection{Harish-Chandra bimodules: sheaf level}

We proceed to give the definitions of sheaf-theoretic HC bimodules. For the sake of concreteness, we will work with the varieties $\cM^{\theta}_{0}$, although the definitions can be stated for any symplectic resolution. So let $\cA^{\theta}_{\lambda}$, $\cA^{\theta}_{\mu}$ be quantizations of $\cM^{\theta}_{0}$. Note that the external tensor product $\cA^{\theta}_{\lambda} \boxtimes (\cA^{\theta}_{\mu})^{\opp}$ is a quantization of the product $\cM^{\theta}_{0} \times \cM^{\theta}_{0}$, where the Poisson structure on the right copy of $\cM^{\theta}_{0}$ has been multiplied by $-1$.

Now let $f: \cM^{0}_{0} \rightarrow \cM^{0}_{0}$ be a $\CC^{\times}$-equivariant Poisson automorphism. Let $\Sigma_{f} \subseteq \cM^{0}_{0} \times \cM^{0}_{0}$ denote the graph of $f$, with its reduced scheme structure, and let $\cS_{f}$ be the \underline{scheme-theoretic} preimage of $\Sigma_{f}$ under the natural map $\cM^{\theta}_{0} \times \cM^{\theta}_{0} \twoheadrightarrow \cM^{0}_{0} \times \cM^{0}_{0}$. For example, when $f$ is the identity then $\Sigma_{f}$ is just the diagonal and $\cS_{f}$ is the usual Steinberg variety. So we call $\cS_{f}$ the \emph{$f$-Steinberg variety.}

\begin{definition}
An $\cA^{\theta}_{\lambda} \boxtimes (\cA^{\theta}_{\mu})^{\opp}$-module $\cB$ is said to be \emph{$f$-twisted Harish-Chandra} (shortly, \emph{$f$-HC}) if it admits a filtration such that $\gr\cB$ is coherent and  \underline{\emph{scheme-theoretically}} supported on the $f$-Steinberg variety $\cS_{f}$.  We denote this category of $f$-HC $\cA^{\theta}_{\lambda} \boxtimes (\cA^{\theta}_{\mu})^{\opp}$-modules by $^{f}\!\gHC(\lambda, \mu)$. 
\end{definition}

Note that 
\[\dim\cS_{f} = \dim\Sigma_{f} = \dim \cM^{0}_{0} = (1/2)\dim(\cM^{\theta}_{0} \times \cM^{\theta}_{0}).\]
Indeed, the last three equalities follow from definition. The proof of the first equality is similar to the proof that the resolution $\pi: \cM^{\theta}_{0} \to \cM^{0}_{0}$ is semismall. Indeed, since $f: \cM^{0}_{0} \to \cM^{0}_{0}$ is a $\C^{\times}$-equivariant Poisson automorphism, the map $f \circ \pi: \cM^{\theta}_{0} \to \cM^{0}_{0}$ is a symplectic resolution, that is, we have a symplectic form on $\cM^{\theta}_{0}$ compatible with the Poisson bracket on $\cM^{0}_{0}$ under the map $f \circ \pi$. Let us call this form $\omega_{f}$. Now equip $\cM^{\theta}_{0} \times \cM^{\theta}_{0}$ with the symplectic form $\Omega = p_1^{*}\omega_{\id} - p_2^{*}\omega_{f}$. The restriction of this form to the regular locus of $\cS_f$ vanishes, and thus $\dim(\cS_f) \leq (1/2)\dim(\cM^{\theta}_{0} \times \cM^{\theta}_{0})$. The other inequality is clear. 

By standard results in homological duality (see, for example, \cite[Section 4.2]{BL}) we get the following.

\begin{proposition}\label{prop:hom_duality}
The functor $\DD: \cB \mapsto \mathcal{E}xt^{\dim \cM^{0}_{0}}_{\cA^{\theta}_{\lambda} \boxtimes (\cA^{\theta}_{\mu})^{\opp}}(\cB, \cA^{\theta}_{\lambda} \boxtimes (\cA^{\theta}_{\mu})^{\opp})$ gives an equivalence:
\[
\DD: \, ^{f}\!\gHC(\lambda, \mu) \buildrel \cong \over \longrightarrow \, ^{f^{-1}}\!\gHC(\mu, \lambda)^{\opp}
\]
where $\,^{f^{-1}}\!\gHC(\mu, \lambda)^{\opp}$ is the opposite category to $\!^{f^{-1}}\!\gHC(\mu, \lambda).$
\end{proposition}
\begin{proof}
	Standard results in homological duality say that the functor $\DD$ gives an equivalence $^{f}\!\gHC(\cA^{\theta}_{\lambda} \boxtimes (\cA^{\theta}_{\mu})^{\opp}) \rightarrow \, ^{f}\!\gHC((\cA^{\theta}_{\lambda})^{\opp} \boxtimes \cA^{\theta}_{\mu})^{\opp}$, where the last opp denotes the opposite category. We have now to compose with the isomorphism $(\cA^{\theta}_{\lambda})^{\opp} \boxtimes \cA^{\theta}_{\mu} \rightarrow \cA^{\theta}_{\mu} \boxtimes (\cA^{\theta}_{\lambda})^{\opp}$. At the associated graded level, this induces the automorphism of $\cM^{\theta}_{0} \times \cM^{\theta}_{0}$ that interchanges the factors. So the image of $\cS_{f}$ is $\cS_{f^{-1}}$ and the result is proved.
\end{proof}

\subsubsection{Quantization of line bundles}

Let us give an example of a Harish-Chandra $\gHC(\lambda, \lambda + \theta)$-bimodule. Recall that on the variety $\cM^{\theta}_{0}$ we have the $\CC^{\times}$-equivariant ample line bundle $\fS(\theta)$. According to \cite[Section 5.1]{BPW}, there exists a unique $(\cA^{\theta}_{\lambda} \boxtimes  (\cA^{\theta}_{\lambda + \theta})^{\opp})$-module $_{\lambda}\cA^{\theta}_{\lambda+\theta}$ admitting a filtration whose associated graded coincides with the pushforward of the line bundle $\fS(\theta)$ to $\cM^{\theta}_{0} \times \cM^{\theta}_{0}$ under the diagonal embedding. Thus, $_{\lambda}\cA^{\theta}_{\lambda + \theta} \in \gHC(\lambda, \lambda + \theta)$. The following proposition now follows from uniqueness of the quantizations of line bundles and from the fact that, it being a quantization of a line bundle, taking tensor product with $_{\lambda}\cA^{\theta}_{\lambda + \theta}$ does not affect the support of a module.

\begin{proposition}\label{prop:translation_equivalences}
For any $n, m \in \ZZ$, taking tensor product with quantizations of line bundles gives equivalences 
\[
^{f}\!\gHC(\lambda+ n\theta, \mu) \leftarrow \, ^{f}\!\gHC(\lambda, \mu) \rightarrow \, ^{f}\!\gHC(\lambda, \mu + m\theta)
\]
\end{proposition}

\subsubsection{Localization theorems}
Note that we have the global sections and localization functors
\[
\Gamma^{\theta}_{\lambda}: \cA^{\theta}_{\lambda}\modd \leftrightarrow \cA_{\lambda}\modd: \Loc^{\theta}_{\lambda}
\]

When these $\Gamma^{\theta}_{\lambda}$ and $\Loc^{\theta}_{\lambda}$ are quasi-inverse equivalences of categories, we say that \emph{abelian localization holds at $\lambda$}. In general, it is not easy to find the locus where abelian localization holds. However, the following result, \cite{BPW}, tells us that abelian localization holds for $\lambda$ sufficiently dominant.

\begin{theorem}[Corollary B.1, \cite{BPW}]\label{thm:bpw_loc}
	Let $\lambda \in \fp$. Then, abelian localization holds at $\lambda + n\theta$ for $n \gg 0$. 
	
\end{theorem}

Since $\cS_{f}$ is defined to be precisely the scheme-theoretic preimage of the graph $\Sigma_{f}$, \cite[Proposition 2.13]{BLPW} implies.

\begin{proposition}\label{prop:loc_theorem}
Global sections and localization induce functors
\[
\Gamma: \, ^{f}\!\gHC(\lambda, \mu) \leftrightarrow \, ^{f^{*}}\!\HC(\lambda, \mu) : \Loc
\]

\noindent moreover, if abelian localization holds at $\lambda$ and $-\mu$, these are quasi-inverse equivalences of categories.
\end{proposition}

\subsection{Duality}
We use our previous work to construct functors between categories of twisted Harish-Chandra bimodules. These functors will be constructed as composition of several equivalences we have already seen. The duality step happens at the level of sheaves, this is just the homological duality given by Proposition \ref{prop:hom_duality}. To pass from sheaves to bimodules we use, of course, localization theorems. The problem here is that, in general, localization does not hold at $\lambda$ and $-\lambda$ simultaneously. So we need to first use an equivalence provided by Proposition \ref{prop:twist}. This will introduce a twist by $\Upsilon$ that will be cancelled at the end using an equivalence of the same form. 

\begin{theorem}\label{thm:duality}
Let $\lambda$ be a period of quantization such that abelian localization holds at $\lambda$. Then, for any $\CC^{\times}$-equivariant Poisson automorphism $f: \cM^{0}_{0} \rightarrow \cM^{0}_{0}$ and $n \gg 0$ there is an equivalence of categories
\[
\DD:\, ^{f^{*}}\!\HC(\lambda, \lambda) \buildrel \cong \over \longrightarrow \, ^{(f^{*})^{-1}}\!\HC(\lambda - n\theta, \lambda - n\theta)^{\opp}
\]
\end{theorem}
\begin{proof}
As we have said above, this equivalence is just a composition of several equivalences that we have introduced before. Let us list these.

\begin{enumerate}
\item[(1)] $^{f^{*}}\!\HC(\lambda, \lambda) \buildrel \cong \over \longrightarrow \, ^{\Upsilon^{*} \circ f^{*}}\!\HC(\lambda, -\lambda)$.
\item[(2)] $^{\Upsilon^{*} \circ f^{*}}\!\HC(\lambda, -\lambda) \buildrel \cong \over \longrightarrow \, ^{f \circ \Upsilon}\!\gHC(\lambda, -\lambda)$.
\item[(3)] $^{f \circ \Upsilon}\!\gHC(\lambda, -\lambda) \buildrel \cong \over \longrightarrow \, ^{\Upsilon \circ f^{-1}}\!\gHC(-\lambda, \lambda)^{\opp}$.
\item[(4)] $^{\Upsilon \circ f^{-1}}\!\gHC(-\lambda, \lambda)^{\opp} \buildrel \cong \over \longrightarrow \, ^{\Upsilon \circ f^{-1}}\!\gHC(-\lambda + n\theta, \lambda - n\theta)^{\opp}$.
\item[(5)] $^{\Upsilon \circ f^{-1}}\!\gHC(-\lambda + n\theta, \lambda - n\theta)^{\opp} \buildrel \cong \over \longrightarrow \, ^{(f^{*})^{-1} \circ \Upsilon^{*}}\!\HC(-\lambda + n\theta, \lambda - n\theta)^{\opp}$.
\item[(6)] $^{(f^{*})^{-1} \circ \Upsilon^{*}}\!\HC(-\lambda + n\theta, \lambda - n\theta)^{\opp} \buildrel \cong \over \longrightarrow \, ^{(f^{*})^{-1}}\!\HC(\lambda - n\theta, \lambda - n\theta)^{\opp}$.
\end{enumerate}

Equivalences (1) and (6) are provided by Proposition \ref{prop:twist}; the equivalence (2) is simply the localization theorem, cf. Proposition \ref{prop:loc_theorem}. The localization theorem also provides the equivalence (5): here, we need to take $n$ large enough so that localization will hold at $-\lambda + n\theta$, cf. Theorem \ref{thm:bpw_loc}. The equivalence (4) is given by the translation equivalences of Proposition \ref{prop:translation_equivalences}. Finally, the duality (3) is the homological duality of Proposition \ref{prop:hom_duality}.
\end{proof}

\begin{remark}\label{rmk:duality2param}
	Theorem \ref{thm:duality} holds, with the same proof, if we take two parameters $\lambda, \mu$ such that localization holds at $\lambda$ and $-\mu$. In this case, we obtain an equivalence $\DD: \HC(\lambda, \mu) \rightarrow \HC(\mu - n\theta, \lambda - n\theta)^{\opp}$, for $n \gg 0$. 
\end{remark}

\begin{remark}\label{rmk:changinglhs}
If $f = \id$ then, in Step (6) of the construction in Theorem \ref{thm:duality}, we could change the right parameter instead of the left parameter. Hence, in this case we also obtain an equivalence $\DD: \HC(\lambda, \lambda) \to \HC(-\lambda + n\theta, -\lambda + n\theta)^{\opp}$.
\end{remark}

\subsection{Connection to rational Cherednik algebras}

Let us remark that the spherical rational Cherednik algebras of the groups $G(\ell, 1, n) = S_{n} \ltimes (\ZZ/\ell\ZZ)^{n}$ can be realized as quantized quiver varieties. Let $Q$ be the cyclic quiver with $\ell$ vertices (in particular, if $\ell = 1$, then $Q$ is a single vertex with a loop). Pick a vertex in $Q$ and call it $0$ (so this is the extending vertex of the affine type A Dynkin diagram underlying $Q$). For a dimension vector $\bv$, we take the vector $n(1, 1, \dots, 1)$ (note that $(1, 1, \dots, 1)$ is the minimal imaginary root for the affine type A Dynkin diagram underlying $Q$) while for a dimension vector $\bw$ we the vector that takes the value $1$ at the vertex $0$ and $0$ everywhere else. For example, for $\ell = 1$ the quiver $\overline{Q^{\heartsuit}}$ with the corresponding dimension vector is

\[
\begin{tikzcd}[every arrow/.append style={shift left}]
n \ar{rr} \ar[,loop, out=123, in=57, distance=0.5cm]{} \ar[,loop, out=237, in=306, distance=0.5cm]{} & & 1 \ar{ll}
\end{tikzcd}
\]

Let us denote by $\bA := eH_{\fp'}e$. This is a $\CC[\fp']$-algebra, where we now denote the space of parameters for the Cherednik algebra by $\fp'$. Let us denote by $\fP$ the set of parameters for the quantum Hamiltonian reduction, $\fP = \fg_{\bv}/[\fg_{\bv}, \fg_{\bv}]$.

\begin{theorem}[\cite{EGGO, gordon_rmk, losev_isoquant, oblomkov}]
	There is a filtered isomorphism $\bA \rightarrow \cA_{\fP}$, which induces a linear isomorphism between the spaces of parameters $\omega: \fp' \rightarrow \fP$. In particular, for every $c \in \fp'$ we have that the algebras $A_{c}$ and $\widehat{\cA}_{\omega(c)}$ are isomorphic as filtered algebras.
\end{theorem}

Let us give a description of the map $\omega$ in the case of rational Cherednik algebra of type $A$. Upon identifying $\fp' = \CC c$, $\fP = \CC z$, this map is simply $c \mapsto -z-1$, see \cite{losev_isoquant}. In particular, we have that the categories $\HC(A_{c}, A_{c'})$ and $\HC(\widehat{\cA}_{-c-1}, \widehat{\cA}_{-c'-1})$ are equivalent. 

Let us now say that, for $c \in \CC$, abelian localization holds for the rational Cherednik algebra $A_{c}$ if it holds for $\lambda$, where $\lambda$ is such that $\cA_{\lambda} = \widehat{\cA}_{-c}$. It is known, see \cite[Sections 5 and 6]{gordon} and \cite[Corollary 4.2]{BE}, that abelian localization fails at $c$ if and only if $c = -r/m$, with $1 < m \leq n$, $r > 0$ and $\gcd(r;m) = 1$. Thus, setting $f = \id$ in Theorem \ref{thm:duality} we get the following result.

\begin{corollary}\label{cor:dualityA}
Let $c \in \CC \setminus \{-r/m : 1 < m \leq n, \gcd(r;m) = 1\}$. Let $N \in \ZZ_{> 0}$ be large enough so that $-c + N \not\in  \{-r/m : 1 < m \leq n, \gcd(r;m) = 1\}$. Then, there is an equivalence of categories
$$
\DD: \HC(c,c) \to \HC(-c + N, -c + N)^{\opp}
$$

\noindent where, recall, $\HC(c,c)$ is the category of HC $H_{c}(S_{n})$-bimodules.
\end{corollary}
\begin{proof}
	Note that, under the assumptions on $c$, the algebras $H_{c}$ and $eH_{c}e$ are Morita equivalent, and so are the algebras $H_{-c+N}$ and $eH_{-c+N}e$. The result now follows from Theorem \ref{thm:duality}.
\end{proof}

 Let us now compare the double wall-crossing bimodule $D_{-c+N}$ with the dual of the regular bimodule $H_{c}$. So let $N$ be sufficiently big so that localization holds at $c$ and $-c+N$. Recall that we denote $D_{-c+N}$ the double wall-crossing bimodule for the algebra $H_{-c+N}$.

\begin{proposition}\label{prop:dcwvsdreg}
Let $c = r/m > 0$, let $N \in \ZZ$ be such that $-c + N > 0$, and consider the duality functor $\DD: \HC(c, c) \rightarrow \HC(-c + N, -c+N)$. Then, $\DD(D_{c}) = H_{-c+N}$ and $\DD(H_{c}) = D_{-c+N}$.
\end{proposition}
\begin{proof}
First of all, note that the duality preserves supports. Since both $H_{c}$ and $H_{-c+N}$ have a unique irreducible HC bimodule with full support, the duality has to send the unique irreducible bimodule with full support over $H_{c}$ to the unique irreducible fully supported bimodule over $H_{-c+N}$. Now, $\DD(D_{c})$ will be an $H_{-c+N}$-bimodule whose socle coincides with the unique minimal ideal of $H_{-c+N}$. Since $H_{-c +N}$ is the injective hull of its unique minimal ideal, we have an embedding $\DD(D_{c}) \rightarrow H_{-c+N}$. That this is an isomorphism now follows by comparing the composition lengths. Now use the fact that $\DD^{2} = \id$ to get the other equality.
\end{proof}

Of course, the previous proposition has its corresponding result when $c$ is aspherical. Here, we need to take $N$ such that localization holds at $-c$ and $N+c$ (note that $N = 1$ suffices). Then, we have the duality functor $\DD: \HC(c, c) \rightarrow \HC(c+N, c+N)^{\opp}$, and we get $\DD(D_{c}(n)) = H_{c +N}(n)$.

\begin{corollary}\label{cor: injective projective}
	For any $c \in \CC$, the regular bimodule $H_{c}$ and the double wall-crossing bimodule $D_{c}$ are injective-projective in the category $\HC(c, c)$.
\end{corollary}
\begin{proof}
We assume $c \in \CC\setminus\RR_{< 0}$. That $H_{c}$ is injective is \cite[Proposition 2.11]{hcbimod}. That $D_{c}$ is injective follows because $D_{c} = \homf{\Delta_{c}(\triv)}{\nabla_{c}(\triv)}$. The result now follows immediately from Proposition \ref{prop:dcwvsdreg}.
\end{proof}

\section{The minimal Serre subcategory containing $H_c$}\label{sect:prblock}

\subsection{Notation and main result}\label{sect6intro}  Assume $c = r/m > 0, \gcd(r;m) = 1$ and $1 < m \leq n$In this section, we describe, via quivers with relations, the minimal Serre subcategory of $\HC(c,c)$ containing the regular bimodule, that we denote by $\cP_{c}$. Note that if $c \neq r/m$ with $\gcd(r;m) = 1$ and $1 < m \leq n$, the category $\cP_{c}$ is equivalent to the category of vector spaces. 

Let us describe some simples in $\cP_{c}$. We recall from \cite{losev_completions} that $H_{c}$ has $\lfloor n/m\rfloor + 2$ two-sided ideals (including $\{0\}$ and $H_{c}$)  that are linearly ordered 

$$
\{0\} \subsetneq \cJ_{0} \subsetneq \cJ_{1} \subsetneq \cdots \subsetneq \cJ_{\lfloor n/m\rfloor - 1} \subsetneq \cJ_{\lfloor n/m \rfloor} = H_{c}.
$$

We will denote $\mathcal{S}_{i} := \cJ_{i}/\cJ_{i-1}$ (where we set $\cJ_{-1} = \{0\}$). This is a simple module in $\cP_{c}$. The support of $\mathcal{S}_{i}$ is

$$
\sing(\cS_{i}) = \overline{\cL_{i}} := \overline{\pi\{(x, y) \in \fh \oplus \fh^{*} : W_{(x,y)} = S_{m}^{\times i} \subseteq S_{n}\}}
$$

\noindent so that, $\cS_{0} = \cJ_{0}$ has full support, while $\cS_{\lfloor n/m\rfloor} = H_{c}/\cJ_{\lfloor n/m\rfloor - 1}$ has minimal support. So we see that $\cP_{c}$  $\lfloor n/m\rfloor + 1$ irreducibles. As the main result of this section, we describe $\cP_{c}$ as an abelian category.

\begin{theorem}\label{thm:prblock}
The let $c = r/m > 0$ with $1 < m \leq n$, $\gcd(r;m) = 1$. The block $\cP_{c}$ is equivalent to the category of representations of the quiver

$$
\xymatrix{\cS_{\lfloor n/m\rfloor} \ar@/^/[rr]^{\alpha_{\lfloor n/m\rfloor}} & & \cS_{\lfloor n/m\rfloor - 1} \ar@/^/[ll]^{\beta_{\lfloor n/m\rfloor}} \ar@/^/[rr]^{\alpha_{\lfloor n/m\rfloor - 1}} & & \cdots \ar@/^/[ll]^{\beta_{\lfloor n/m\rfloor - 1}} \ar@/^/[r] & \cdots \ar@/^/[l] \ar@/^/[r] & \ar@/^/[l] \cS_{1}  \ar@/^/[rr]^{\alpha_{1}} & & \cS_{0}  \ar@/^/[ll]^{\beta_{1}}}  
$$

\noindent with relations $\alpha_{i}\beta_{i} = \beta_{i}\alpha_{i} = 0$ for every $i = 1, \dots, \lfloor n/m\rfloor$. 
\end{theorem}

To prove Theorem \ref{thm:prblock} we have to show first that the Ext quiver of $\cP_{c}$ is exactly as stated. What we can get immediately is that the quiver as in the statement of Theorem \ref{thm:prblock} is a subquiver of the Ext quiver. Indeed, we know by \cite{hcbimod} that $H_{c}$ is injective, so it has to be the injective hull of $\cS_{0}$. Similarly, the double wall-crossing bimodule $D_{c}$ has to be the injective hull of $\cS_{\lfloor n/m\rfloor}$. So we, at least, must have the arrows appearing in Theorem \ref{thm:prblock}. The relations, and the fact that we have no more arrows, are obtained via the vanishing of several extension groups, which in turn is obtained using induction and restriction functors. 

\subsection{Quiver representations}\label{sect:quivers} In this subsection, we recall some results on quiver representations that we will use. This material is standard and can be found in, for example, \cite{assem}. Let $Q$ be a quiver and $\C Q$ its path algebra. We will denote by $A \subseteq \C Q$ the ideal generated  by the arrows in $Q$. A two-sided ideal $I \subseteq \C Q$ is called \emph{admissible} if there exists $k \gg 0$ such that $A^{k} \subseteq I \subseteq A^{2}$. If $I$ is an admissible ideal, we denote $\C (Q, I) := \CC Q/I$ and we say that $(Q, I)$ is a \emph{quiver with relations}. 

It is well-known that every finite-dimensional algebra $R$ is Morita equivalent to $\C(Q, I)$ for an appropriate quiver with relations $(Q, I)$. For $Q$, we can always take the \emph{ext} quiver of $R$. Its vertices are labeled by isomorphism classes of simple $R$-modules, and the number of arrows from $S_{i}$ to $S_{j}$ is $\dim \Ext^{1}_{R}(S_{i}, S_{j})$. To compute the ideal $I$ is more subtle, we will elaborate below on how to do this in the case of interest for us.

For a vertex $i \in Q_{0}$, let us denote by $\varepsilon_{i} \in \C(Q, I)$ the lazy path of length 0 that starts and ends at vertex $i$, and we will denote by $S_{i}$ the simple $\C(Q, I)$-module that is concentrated at vertex $i$. The projective cover $P(i)$ of $S_{i}$ is $P(i) = \CC(Q, I)\varepsilon_{i}$. A basis for $P(i)$ consists of all paths that start at vertex $i$ (with the appropriate relations) and multiplication by an element of $\CC(Q, I)$ is done by concatenating paths. Note that the multiplicity of $S_{i}$ on $P(i)$ is equal to the number of paths from vertex $i$ to itself, including the lazy path. In particular, if the multiplicity of $S_{i}$ on $P(i)$ is $1$ then we get, using the fact that the ideal $I$ of relations is admissible, that every path of positive length starting and ending at $i$ must belong to the ideal of relations $I$. 

For the rest of this section, we will assume that the set of simples, up to isomorphism, in $\C(Q, I)\modd$ is in bijection with the set of vertices in $Q$. If $R$ is a finite-dimensional algebra, the ext quiver of $R$ satisfies this by definition. Let $J \subseteq Q_{0}$ be a subset of vertices, and consider the Serre subcategory $\cT$ of $\C(Q, I)$-mod consisting of modules whose simple subfactors belong to $\{S_{i} : i \in J\}$. Since $\cT$ is a Serre subcategory, we can consider the Serre quotient $\C(Q, I)\modd/\cT$. By definition, it comes equipped with a quotient functor $\pi: \C(Q,I)\modd \to \C(Q,I)\modd/\cT$ that is universal with the property that $\pi(C) = 0$ for every $C \in \cT$ and admits a right inverse (and left adjoint) $\pi^{!}$. In fact, the category $\C(Q,I)\modd/\cT$ has a very explicit description. Consider the projective $P_{J} := \bigoplus_{i \not\in J} P(i)$  and its endomorphism algebra $E := \End_{\C(Q,I)\modd}(P_{J})^{\opp}$. The following result should be well-known, but we could not find a reference for it in the literature.

\begin{proposition}
We can identify $\C(Q, I)\modd/\cT$ with the category $E\modd$ of modules over the endomorphism algebra $E$. The quotient functor $\pi$ is identified with $\Hom_{\C(Q,I)\modd}(P_{J}, \bullet)$, and its right inverse (and left adjoint) with $P_{J} \otimes_{E} \bullet$. 
\end{proposition}
\begin{proof}
First, note that $\pi$ is exact because $P_{J}$ is projective and, by the very definition of $P_{J}$, $\Hom_{\C(Q, I)\modd}(P_{J}, N) = 0$ if and only if $N$ does not have simple subquotients of the form $S(i)$ for $i \in Q_{0} \setminus J$, that is, the kernel of $\pi$ is precisely $\cT$. 

Let us denote $\pi^{!} = P_{J} \otimes_{E} \bullet: E\modd \to \C(Q, I)\modd$. We claim, first, that the unit of adjunction $\id_{E\modd} \to \pi\circ\pi^{!}$ is an isomorphism. Let us remark that this is given by $M \to \Hom_{\C(Q, I)\modd}(P_{J}, P_{J} \otimes_{E} M)$, $m \mapsto (p \mapsto p \otimes_{E} m)$. It is easy to see that this is an isomorphism when $M$ is free. From here, since $\pi^{!}$ is right exact and $\pi$ is exact, it follows that the counit of adjunction is surjective.  It remains to show that it is injective, that is, that for every $m \in M$ the map $p \mapsto p \otimes_{E} m$ is nonzero. This is equivalent to show that, if $E/I$ is a nonzero cyclic left $E$-module, then $P \otimes_{E} (E/I)$ is nonzero. 

To show that $P \otimes_{E} (E/I)$ is nonzero, it suffices to provide a nonzero map $f: P_{J} \otimes_{\CC}(E/I) \to V$, where $V$ is some vector space and $f$ satisfies $f(p \otimes em) = f(pe \otimes m)$ for $e \in E$. We set $X := \sum_{i \in I} \operatorname{im}(i) \subseteq P_{J}$. We claim that $X$ is a proper subspace of $P_{J}$. The modules $P(j), j \in Q_{0} \setminus J$ have a unique maximal submodule, and a map $\varphi: P(j_{1}) \to P(j_{2})$ is surjective if and only if $j_{1} = j_{2}$ and $\varphi$ is an isomorphism. It follows that the map $X \to P(j)$ that is obtained by composing the inclusion $X \hookrightarrow P_{J}$ with the projection $P_{J} \twoheadrightarrow P(j)$ is surjective if and only if $I$ contains an element that, in matrix form $(\varphi_{ij}: P(i) \to P(j))_{i, j \in Q_{0} \setminus J}$ has $\varphi_{jj}$ being an isomorphism. So, if $X = P_{J}$, $I$ must contain an invertible element of $E$, and therefore $I = E$, a contradiction. We set $V := P_{J}/X \neq 0$. Then, we get a nonzero map $f: P_{J} \otimes_{\CC} (E/I) \to V$, $f(p \otimes m) = pm + X$, where $pm$ denotes the action of $m$ on $p$. It is easy to see that $f$ is well-defined, that it is surjective and that it satisfies $f(p \otimes em) = f(pe \otimes m)$. This implies that $P \otimes_{E} (E/I)$ is nonzero. So $\id_{E\modd} \to \pi \circ \pi^{!}$ is an isomorphism, as desired.

Note that it follows that, for every representation $N \in \C(Q, I)\modd$, the kernel and co-kernel of the co-unit of adjunction $\pi^{!} \circ \pi(N) \to N$ belong to $\cT$. Indeed, this is the case because $\pi \circ \pi^{!} \circ \pi \cong \pi$. Also, let $0 \to M_{1} \to M_{2} \to M_{3} \to 0$ be an exact sequence in $E\modd$. Since $\pi^{!}$ is right exact, we have an exact sequence $0 \to K \to \pi^{!}(M_{1}) \to \pi^{!}(M_{2}) \to \pi^{!}(M_{3}) \to 0$ in $\C(Q, I)\modd$. Applying $\pi$ we get an exact sequence in $E\modd$. So $\pi(K) = 0$, in other words, $K \in \cT$. 

 Now let $\omega: \C(Q, I)\modd \to \cD$ be an exact functor such that $\omega(N) = 0$ for $N \in \cT$. From the results of the previous paragraph, we have $\omega \cong (\omega \circ \pi^{!}) \circ \pi$, and the functor $\omega \circ \pi^{!}: E\modd \to \cD$ is exact. This implies the result.
\end{proof}


The endomorphism algebra $E$ can also be described via quivers with relations starting from $(Q, I)$ as follows. We have that $E = \bigoplus_{i_{1}, i_{2} \not\in J} \varepsilon_{i_{1}}\C(Q, I)\varepsilon_{i_{2}}$. So $E = (Q', I')$, where the vertices of $Q'$ are $Q_{0}\setminus J$, the arrows in $Q'$ are in bijection with paths $i_1 \to j_1 \to \cdots \to j_s \to i_2$, where $i_1, i_2 \in Q_0 \setminus J$ and $j_1, \dots, j_s \in J$, $s \geq 0$. The ideal $I'$ is obtained from $I$. We warn, however, that the relations $I'$ may not be admissible. 

\begin{example}
Consider the quiver

$$
Q = \xymatrix{ 1 \ar@/^/[rr]^{\alpha_{1}} & & 2 \ar@/^/[rr]^{\alpha_{2}} \ar@/^/[ll]^{\beta_{1}} & & 3 \ar@/^/[ll]^{\beta_{2}}}
$$

\noindent with relations $\alpha_{1}\beta_{1} = \beta_{2}\alpha_{2}$, $\beta_{1}\alpha_{1} = 0$, $\alpha_{2}\beta_{2} = 0$. Now consider $\cT = \langle S_{1}, S_{3}\rangle$. We get that the category $\C(Q,I)\modd/\cT$ is the category of representations of the quiver 
$$
\xymatrix{2 \ar@(ld, lu)^{\alpha_{1}\beta_{1}} \ar@(ru, rd)^{\beta_{2}\alpha_{2}}}
$$

\noindent with relations $(\alpha_{1}\beta_{1}) = (\beta_{2}\alpha_{2})$, $(\alpha_{1}\beta_{1})^{2} = 0$. These relations are not admissible, but they allow us to see that $\C(Q,I)\modd/\cT$ is equivalent to the category of representations of the algebra of dual numbers.
\end{example}

\subsection{Duality} We come back to the study of the Serre subcategory $\cP_{c} \subseteq \HC(c,c)$. Let us show that the duality functor $\DD$ fixes the category $\cP$, in the sense that it induces a functor $\DD: \cP_{c} \to \cP_{-c+N}^{\opp}$.

\begin{lemma}\label{lemma:dualitysimples}
Consider the functor $\DD: \HC(c,c) \to \HC(-c+N, -c+N)^{\opp}$. Then, $\DD(\cS_{i}) = \cS_{i}$ for every $i = 0, \dots, \lfloor n/m\rfloor$. Thus, we have an induced duality functor $\DD: \cP_{c} \to \cP_{-c+N}$.
\end{lemma}
\begin{proof}
This follows from Proposition \ref{prop:dcwvsdreg}. We have that $H_{c}$ is filtered
\[
H_{c} = \cJ_{\lfloor n/m \rfloor} \supsetneq \cdots \supsetneq \cJ_{1} \supsetneq \cJ_{0} \supsetneq \{0\}
\]
with succesive subquotients being $\cS_{i} = \cJ_{i}/\cJ_{i-1}$. Thus, we obtain:
\[
\DD(H_{c}) \buildrel {\varphi_1} \over \twoheadrightarrow \DD(\cJ_{\lfloor n/m \rfloor - 1}) \twoheadrightarrow \cdots \twoheadrightarrow \DD(\cJ_{1}) \buildrel \varphi_{\lfloor n/m \rfloor} \over \twoheadrightarrow \DD(\cJ_{0}) \twoheadrightarrow \{0\}.
\]
and $\DD(H_c)$ is filtered by
\[
\DD(H_{c}) \supsetneq \ker(\varphi_{\lfloor n/m\rfloor} \circ \cdots \circ \varphi_{1}) \supsetneq \cdots \supsetneq \ker(\varphi_1) \supsetneq \{0\} 
\]
with succesive subquotients being $\DD(\cS_0), \DD(\cS_1), \dots, \DD(\cS_{\lfloor n/m\rfloor})$. Since $\DD(H_{c}) = D_{-c+N}$ which is uniserial with succesive quotients being $\cS_0, \cS_1, \dots, \cS_{\lfloor n/m\rfloor}$, we obtain the result. 
\end{proof}

\subsection{Projective covers in quotient categories.} The regular bimodule $H_{c}$ is the projective cover of $\cS_{\lfloor n/m\rfloor}$. We cannot expect the ideal $\cJ_{i}$ to be the projective cover of $\cS_{i}$, but this will hold in an appropriate quotient category. Let us denote by $\HC^{i}(c,c)$ the quotient category $\HC(c,c)/\HC_{\overline{\cL_{i+1}}}(c,c)$. By $\pi: \HC(c,c) \to \HC^{i}(c,c)$ we will denote the quotient functor. Note, in particular, that the inclusion $\cJ_{i} \hookrightarrow H_{c}$ induces an isomorphism $\pi(\cJ_{i}) \cong \pi(H_{c})$. 

Recall that, when $j \leq i$, the functor $\bullet_{\dagger_{j}} := \bullet_{\dagger^{S_{n}}_{S_{m}^{\times j}}}$ factors through $\HC^{i}(c,c)$. We will abuse the notation here, and write $\bullet_{\dagger_{j}}: \HC^{i}(c,c) \to \underline{\HC}^{\Xi}(c,c)$. This functor admits a right adjoint, which is simply $\pi \circ \bullet^{\dagger_{j}}$.  

\begin{lemma}\label{lemma:projquotient}
The bimodule $\pi(H_{c})$ is projective in the quotient category $\HC^{i}(c,c)$.
\end{lemma}
\begin{proof}
We have to show that $\Ext^{1}_{\HC^{i}}(\pi(H_{c}), \pi(B)) = 0$ for any simple bimodule $B$ whose support is not contained in $\overline{\cL_{i+1}}$. Pick such a simple bimodule $\pi(B)$ and consider a short exact sequence 
\begin{equation}\label{eqn:ses}
0 \to \pi(B) \to \pi(X) \to \pi(H_{c}) \to 0
\end{equation}

\noindent in $\HC^{i}$ where, for brevity, we denote $\HC^{i} := \HC^{i}(c,c)$. We separate in two cases. 

{\it Case 1. $B$ has full support.} In this case, we must have $B = \cS_{0}$. The restriction functor $\bullet_{\dagger_{0}}$ factors through the quotient category $\HC^{i}$ and so we can apply it to our exact sequence. The resulting short exact sequence splits, so we have $X_{\dagger} = (H_{c})_{\dagger} \oplus (\cS_{0})_{\dagger}$. By adjunction, we get a map $X \to H_{c} \oplus H_{c}$ whose kernel has proper support. But from (\ref{eqn:ses}) we can see that $\pi(X)$ does not have nonzero subobjects of the form $\pi(B')$, where $B'$ is a simple with proper support. So the map $X \to H_{c} \oplus H_{c}$ induces an embedding $\pi(X) \hookrightarrow \pi(H_{c}) \oplus \pi(H_{c})$. From here, it is easy to see that we must have $X \cong \pi(B) \oplus \pi(H_{c})$.

{\it Case 2. $B$ has proper support.} Say $\sing(B) = \overline{\cL_{j}}$ for some $0 < j < i+1$, and apply the restriction functor $\bullet_{\dagger_{j}}$ to the sequence (\ref{eqn:ses}), so that $B_{\dagger_{j}}$ is finite-dimensional. Now the sequence $0 \to B_{\dagger_{j}} \to X_{\dagger_{j}} \to H_{c, \dagger_{j}} \to 0$ splits, since $H_{c, \dagger_{j}}$ is projective in the category of equivariant bimodules. So we get the adjunction map $X \to (X_{\dagger_{j}})^{\dagger_{j}} = (B_{\dagger_{j}})^{\dagger_{j}} \oplus H_{c}$. Now we note that $\pi(X)$ does not have subobjects that are killed by the restriction functor $\bullet_{\dagger_{j}}$. Indeed, this follows because of our choice of $j$ and because every subobject of $H_{c}$ has full support. It follows, in particular, that the map $\pi(X) \to \pi(B_{\dagger_{j}}^{\dagger_{j}}) \oplus \pi(H_{c})$ induced by the adjunction morphism is injective. 

Now, by \cite{losev_completions}, $(B_{\dagger_{j}})^{\dagger_{j}}$ does not have subbimodules with full support, while it is clear that $H_{c}$ does not have subbimodules with support $\overline{\cL_{j}}$. It follows that none of the morphisms $\pi(X) \to \pi((B_{\dagger_{j}})^{\dagger_{j}}) \oplus \pi(H_{c}) \to  \pi((B_{\dagger_{j}})^{\dagger_{j}})$, $\pi(X) \to \pi((B_{\dagger_{j}})^{\dagger_{j}}) \oplus \pi(H_{c}) \to \pi(H_{c})$ is zero. By (\ref{eqn:ses}), the second one of these maps must be surjective. But now it follows that the image of the map $\pi(X) \to \pi((B_{\dagger_{j}})^{\dagger_{j}})$ concides with $\pi(B)$, and this provides a splitting to the inclusion $\pi(B) \to X$. We are done.
\end{proof}

\begin{corollary}\label{cor:noextdown}
Let $B \in \HC(c,c) \setminus \HC_{\overline{\cL_{i+1}}}(c,c)$ be simple. Then, $\Ext^{1}_{H_{c}\bimod}(B, \cS_{i+1}) = 0 = \Ext^{1}_{H_{c}\bimod}(\cS_{i+1}, B)$ unless $B = \cS_{i}$, in which case both extension groups are $1$-dimensional.
\end{corollary}
\begin{proof}
Note that $\Ext^{1}_{H_{c}\bimod}(B, \cS_{i+1})$ is embedded in $\Ext^{1}_{\HC^{i}}(\pi(B), \pi(\cS_{i+1}))$. The projective cover of $\pi(\cS_{i+1})$ in $\HC^{i}$ is $\pi(H_{c})$, which has a unique composition series whose simple subquotients are $\cS_{0}, \dots, \cS_{i}, \cS_{i+1}$. It follows from the computation of $\Ext^{1}$ using projective resolutions, that $\Ext^{1}_{\HC^{i}}(B, \cS_{i}) \neq 0$ implies $B = \cS_{i}$, and $\Ext^{1}_{\HC^{i}}(\cS_{i}, \cS_{i+1})$ is $1$-dimensional.

Note that the duality functor induces a duality on $\HC^{i}$ which fixes both $\cS_{i}$ and $\cS_{i+1}$, cf. Lemma \ref{lemma:dualitysimples}. Thus, $\Ext^{1}_{\HC^{i}}(\cS_{i+1}, B) \neq 0$ if and only if $\Ext^{1}_{\HC^{i}}(\DD(B), \cS_{i+1}) \neq 0$, if and only if $\DD(B) \cong \cS_{i}$, which happens if and only if $B \cong \cS_{i}$. The result follows. 
\end{proof}

\begin{corollary}\label{cor: ext simples above}
Assume $\Ext^{1}(\cS_{i}, \cS_{j}) \neq \{0\}$. Then, $|i - j| = 1$. 
\end{corollary}
\begin{proof}
First note that since simples in category $\cO_{c}$ do not have self-extensions, the same is true for simples in the category $\HC$. By Lemma \ref{lemma:dualitysimples}, $\Ext^{1}(\cS_{i}, \cS_{j}) \neq 0$ if and only if $\Ext^{1}(\cS_{j}, \cS_{i}) \neq 0$. So we may assume $i < j$, in which case the result follows from Corollary \ref{cor:noextdown}.
\end{proof}

\subsection{Proof of Theorem \ref{thm:prblock}} It follows from Lemma \ref{lemma:dualitysimples} and Corollary \ref{cor:noextdown} that the ext quiver of the block $\cP_{c}$ is precisely

$$
Q := \xymatrix{\cS_{\lfloor n/m\rfloor} \ar@/^/[rr]^{\alpha_{\lfloor n/m\rfloor}} & & \cS_{\lfloor n/m\rfloor - 1} \ar@/^/[ll]^{\beta_{\lfloor n/m\rfloor}} \ar@/^/[rr]^{\alpha_{\lfloor n/m\rfloor - 1}} & & \cdots \ar@/^/[ll]^{\beta_{\lfloor n/m\rfloor - 1}} \ar@/^/[r] & \cdots \ar@/^/[l] \ar@/^/[r] & \ar@/^/[l] \cS_{1}  \ar@/^/[rr]^{\alpha_{1}} & & \cS_{0}  \ar@/^/[ll]^{\beta_{1}}}  
$$

\noindent and our job now is to find the relations. Since $H_{c}$ is the projective cover of $\cS_{\lfloor n/m\rfloor}$ and this simple appears with multiplicity 1 in $H_{c}$, every path in $Q$ that starts and ends in $\cS_{\lfloor n/m\rfloor}$, except for the lazy path, must be zero. In particular, we get that $\beta_{\lfloor n/m\rfloor}\alpha_{\lfloor n/m\rfloor} = 0$. Now we consider the quotient category $\cP_{c}/\langle \cS_{\lfloor n/m\rfloor}\rangle$. From the general theory of quiver representations, cf. Section \ref{sect:quivers}, this quotient category is equivalent to the representations of the quiver

$$
\xymatrix{\cS_{\lfloor n/m\rfloor - 1} \ar@(ul, ur) \ar@/^/[rr]^{\alpha_{\lfloor n/m\rfloor - 1}} & & \cdots \ar@/^/[ll]^{\beta_{\lfloor n/m\rfloor - 1}} \ar@/^/[r] & \cdots \ar@/^/[l] \ar@/^/[r] & \ar@/^/[l] \cS_{1}  \ar@/^/[rr]^{\alpha_{1}} & & \cS_{0}  \ar@/^/[ll]^{\beta_{1}}}  
$$

\noindent where the loop at $\cS_{\lfloor n/m\rfloor - 1}$ represents the path $\alpha_{\lfloor n/m\rfloor}\beta_{\lfloor n/m\rfloor}$. Thanks to Lemma \ref{lemma:projquotient}, the projective cover of $\cS_{\lfloor n/m\rfloor - 1}$ in this category is (the projection of) $H_{c}$. Since $\cS_{\lfloor n/m\rfloor - 1}$ appears with multiplicity $1$ here, we get that the loop $(\alpha_{\lfloor n/m\rfloor}\beta_{\lfloor n/m\rfloor})$ and the 2-cycle $\beta_{\lfloor n/m\rfloor -1}\alpha_{\lfloor n/m\rfloor -1}$ are zero. Inductively, we conclude that $\alpha_{i}\beta_{i} = 0 = \beta_{i}\alpha_{i}$ for every $i$. \\

Now we have to show that there are no relations involving only $\alpha$'s or involving only $\beta$'s. That there are no relations involving only $\alpha$'s follows because the regular bimodule $H_{c}$ corresponds to the representation

$$
\xymatrix{\CC \ar@/^/[rr]^{1} & & \CC \ar@/^/[ll]^{0} \ar@/^/[rr]^{1} & & \cdots \ar@/^/[ll]^{0} \ar@/^/[r]^{1} & \cdots \ar@/^/[l]^{0} \ar@/^/[r]^{1} & \ar@/^/[l]^{0} \CC  \ar@/^/[rr]^{1} & & \CC  \ar@/^/[ll]^{0}}  
$$

\noindent and, similarly, that there are no relations involving only $\beta$'s can be seen from the existence of the double wall-crossing bimodule $D_{c}$. This finishes the proof of Theorem \ref{thm:prblock}. \hfill $\square$ \\

By Lemma \ref{lemma:dualitysimples}, the duality functor preserves the principal block. It is easy to see from Proposition \ref{prop:dcwvsdreg} that, on the level of the quiver $Q$ above, the duality fixes every vertex and it behaves on arrows as $\alpha_{i} \leftrightarrow \beta_{i}$. 

\subsection{Tensor structure} While Theorem \ref{thm:prblock} describes the structure of  $\cP_{c}$ as an abelian category, it gives no information on the tensor structure that is inherited from the tensor structure on $\HC(c,c)$. Let us first define such tensor structure. Let $i: \cP_{c} \to \HC(c,c)$ be the inclusion, with left adjoint $\,^{!}i: \HC(c,c) \to \cP_{c}$, $\,^{!}i(B)$ is the largest quotient of $B$ that belongs to $\cP_{c}$. For $B_1, B_2 \in \cP_{c}$ define
\[
B_1 \tilde{\otimes} B_2 = \, ^{!}i(B_1 \otimes_{H_c} B_2) \in \cP_{c}.
\]
The main goal of this section is to describe the tensor product $\cS_{i} \tilde{\otimes} \cS_{j}$ of simple bimodules in $\cP_{c}$. We start with the following preparatory lemma, which is an easy consequence of \cite[Corollary 2.7]{hcbimod}.

\begin{lemma}\label{lemma:tensor diff supp}
Assume $i \neq j$. Then, $\cS_{i} \otimes_{H_{c}} \cS_{j} = 0$.
\end{lemma}
\begin{proof}
Since $\cS_{i}$ and $\cS_{j}$ have different support, the result follows from \cite[Corollary 2.7]{hcbimod}. 
\end{proof}

The next step is to examine the tensor product of quotients of the regular bimodule $H_{c}$ with the double wall-crossing bimodule $D$. We will do prove slightly more general statement and examine the tensor products of quotients of $H_{c}$ with sub-bimodules of $D$. Let us, first, set some notation. For each $j = 0, \dots, \lfloor n/m\rfloor$, we will denote by $\cI_{j}$ the unique sub-bimodule of $D$ that surjects onto $\cS_{j}$. For example, $\cI_{\lfloor n/m\rfloor} = \cS_{\lfloor n/m\rfloor}$ while $\cI_{0} = D$. 

\begin{lemma}\label{lemma:sub quot zero}
Let $\cJ_{k}$ be a nonzero two-sided ideal of $H_{c}$. Then, $ (H_{c}/\cJ_{k}) \otimes_{H_{c}} D = 0$. Moreover, $(H_{c}/\cJ_{k}) \otimes_{H_{c}} \cI_{j} = 0$ for $k \geq j$. 
\end{lemma}
\begin{proof}
We tensor the short exact sequence $0 \to \cJ_{k} \to H_{c} \to H_{c}/\cJ_{k} \to 0$ with $D$ on the right to get
$$
\cJ_{k} \otimes_{H_{c}} D \to D \to (H_{c}/\cJ_{k}) \otimes_{H_{c}} D \to 0,
$$

\noindent so $(H_{c}/\cJ_{k}) \otimes_{H_{c}}D$ is a quotient of $D$ which is annihilated on the left by $\cJ_{k}$. But every nonzero quotient of $D$ has $\cJ_{0}$ as a quotient, and the left annihilator of the latter bimodule is trivial. Thus, $ (H_{c}/\cJ_{k}) \otimes_{H_{c}} D$ must be zero. The proof of the statement concerning $\cI_{k}$ is similar, noting that every nonzero quotient of $\cI_{k}$ surjects onto $\cS_{k} = \cJ_{k}/\cJ_{k-1}$, and this bimodule is not annihilated by $\cJ_{k}$, cf. \cite[Theorem 5.8.1]{losev_completions}.
\end{proof}
\begin{theorem}\label{thm:tensor simples}
We have

$$
\cS_{i} \tilde{\otimes} \cS_{j} = \begin{cases} \cI_{j} & \text{if i = j}, \\ 0 & \text{else}. \end{cases}
$$
\end{theorem}
\begin{proof}
That $\cS_{i} \otimes \cS_{j} = 0$ if $i \neq j$ is Lemma \ref{lemma:tensor diff supp} above. 
We move on to the $i = j$ case. We tensor the short exact sequence $0 \to \cI_{j+1} \to \cI_{j} \to \cS_{j} \to 0$ with $\cS_{j}$ on the right to get
$$
\cI_{j+1} \otimes_{H_{c}} \cS_{j} \to \cI_{j} \otimes_{H_{c}} \cS_{j} \to \cS_{j} \otimes_{H_{c}} \cS_{j} \to 0.
$$
Note that $\cI_{j+1} \otimes_{H_{c}} \cS_{j}$, this follows from Lemma \ref{lemma:tensor diff supp} and the fact that the composition series of $\cI_{j+1}$ does not include $\cS_{j}$. So $\cS_{j} \otimes_{H_{c}} \cS_{j} \cong \cI_{j} \otimes_{H_{c}} \cS_{j}$. 

Now we consider the short exact sequence $0 \to \cS_{j} \to H_{c}/\cJ_{j-1} \to H_{c}/\cJ_{j} \to 0$, and we tensor it with $\cI_{j}$ to get
$$
\cS_{j} \otimes_{H_{c}} \cI_{j} \to (H_{c}/\cJ_{j-1}) \otimes_{H_{c}} \cI_{j} \to (H_{c}/\cJ_{j}) \otimes_{H_{c}} \cI_{j} \to 0
$$ 

By Lemma \ref{lemma:sub quot zero}, the bimodule $ (H_{c}/\cJ_{j}) \otimes_{H_{c}} \cI_{j}$ is zero. Moreover, note that the (left or right) action of $H_{c}$ on $\cI_{j}$ factors through $H_{c}/\cJ_{j-1}$. So we get:
$$
(H_{c}/\cJ_{j-1}) \otimes_{H_{c}} \cI_{j} = (H_{c}/\cJ_{j-1}) \otimes_{H_{c}/\cJ_{j-1}} \cI_{j} = \cI_{j},
$$
\noindent to conclude that $\cS_{j} \otimes_{H_{c}} \cS_{j} \cong \cS_{j} \otimes_{H_{c}} \cI_{j}$ surjects onto $\cI_{j}$. Thus, $\cS_j \tilde{\otimes} \cS_j$ surjects onto $\cI_j$.

Now, both objects $\cS_{j} \tilde{\otimes} \cS_{j}$ and $\cI_{j}$ belong to the Serre subcategory of $\cP_{c}$ spanned by simples with support inside $\overline{\cL_{j}}$. By Theorem \ref{thm:prblock} the bimodule $\cI_{j}$ is projective there. So the surjection $\cS_{j} \tilde{\otimes} \cS_{j} \to \cI_{j}$ must split and we have
$$
\cS_{j} \tilde{\otimes} \cS_{j} = \cI_{j} \oplus X
$$

\noindent for some bimodule $X \in \HC_{\overline{\cL_{j}}}$. Note that $X$ cannot contain $\cS_{j}$ as a composition factor, this follows since $\HC_{\overline{\cL_{j}}}/\HC_{\overline{\cL_{j+1}}}$ is a monoidal category, equivalent to the category of representations of the symmetric group $S_{j}$, with (the image of) $\cS_{j}$ being the unit object, cf. \cite[Theorem 6.8]{hcbimod}. If $X$ has a quotient of the form $\cS_{k}$ for $k > j$ then we have $X \otimes_{H_{c}} \cS_{k} \neq 0$, which contradicts $\cS_{j} \otimes_{H_{c}} \cS_{j} \otimes_{H_{c}} \cS_{k} = 0$. Thus, $X = 0$ and we are done.
\end{proof}

Note that as an immediate consequence of Theorem \ref{thm:tensor simples} we obtain another expression for the double wall-crossing bimodule.

\begin{corollary}
We have $D = \cS_{0} \tilde{\otimes} \cS_{0} = \cS_{0} \otimes_{H_{c}} \cS_{0}$, where $\cS_{0} = \cJ_{0}$ is the unique minimal two-sided ideal of $H_{c}$.
\end{corollary}
\begin{proof}
The statement $D = \cS_{0} \tilde{\otimes} \cS_{0}$ is a special case of Theorem \ref{thm:tensor simples}. To verify that, in fact, $D = \cS_0 \otimes_{H_c} \cS_0$ note that $D$ is already projective in the category $\HC(c,c)$, cf. Corollary \ref{cor: injective projective}, so we can apply verbatim the argument in the last paragraph of the proof of Theorem \ref{thm:tensor simples},
\end{proof}

Thanks to Lemma \ref{lemma:tensor diff supp} we obtain
$$
D \otimes_{H_{c}} \cS_{k} = 0 \; \text{for} \; k > 0.
$$

\noindent In particular, the duality functor $\DD$ does not intertwine the tensor products.

 \end{document}